\documentclass[11pt]{amsart}
\usepackage{geometry}                
\usepackage{graphicx}
\usepackage{amssymb}
\usepackage{epstopdf}
\newtheorem{theorem}{Theorem}[section]
\newtheorem{lemma}[theorem]{Lemma}
\newtheorem{corollary}[theorem]{Corollary}

\theoremstyle{definition}

\theoremstyle{remark}
\newtheorem{remark}[theorem]{Remark}

\numberwithin{equation}{section}

\newcommand{\p}{\partial}
\newcommand{\wt}{\widetilde}

\newcommand{\CM}{\mathcal{M}}
\newcommand{\CH}{\mathcal{H}}

\newcommand{\CF}{\mathcal{F}}

\newcommand{\BFR}{\mathbf{R}}
\newcommand{\BFN}{\mathbf{N}}
\newcommand{\BFC}{\mathbf{C}}
\newcommand{\BFG}{\mathbf{G}}
\newcommand{\CN}{\mathcal{N}}
\newcommand{\CC}{\mathcal{C}}
\newcommand{\BFE}{\mathbf{E}}
\newcommand{\CD}{\mathcal{D}}

\newcommand{\BFT}{\mathbf{T}}
\newcommand{\ep}{\epsilon}
\newcommand{\BBN}{\mathbb{N}}
\newcommand{\BBH}{\mathbb{H}}
\newcommand{\vr}{\vec{r}}
\newcommand{\BFQ}{\mathbf{Q}}
\newcommand{\BFL}{\mathbf{L}}
\newcommand{\BFJ}{\mathbf{J}}

\newcommand{\be}{\begin{equation} }
\newcommand{\ee}{\end{equation}}
\newcommand{\bse}{\begin{subequations}}
\newcommand{\ese}{\end{subequations}}

\begin{document}

\title[Concentrated steady vorticities]{Concentrated steady vorticities of the Euler equation on 2-d domains and their linear stability}                  

 \author{Yiming Long${}^\#$}
 \address{Chern Institute of Mathematics and LPMC, Nankai University}
 \email{longym@nankai.edu.cn}
\thanks{${}^\#$ Partially supported by NSFC Grants (Nos. 11131004, 11671215 and 11790271), LPMC of Ministry of Education of China, Nankai University, and the Beijing Advanced Innovation Center for Imaging Technology at Capital Normal University.}

\author{Yuchen Wang}     
\address{Chern Institute of Mathematics, Nankai University}
\email{wangyuchen@mail.nankai.edu.cn}    

 \author{Chongchun Zeng${}^\dagger$}
 \address{School of Mathematics, Georgia Institute of Technology and Chern Institute of Mathematics, Nankai University}
\email{chongchun.zeng@math.gatech.edu}
\thanks{${}^\dagger$ Partially supported by NSF-DMS 1362507}

\begin{abstract}
We consider concentrated vorticities for the Euler equation on a smooth domain $\Omega \subset \BFR^2$ in the form of 
\[
\omega = \sum_{j=1}^N \omega_j \chi_{\Omega_j}, \quad |\Omega_j| = \pi r_j^2, \quad \int_{\Omega_j} \omega_j d\mu =\mu_j \ne 0,
\] 
supported on well-separated vortical domains $\Omega_j$, $j=1, \ldots, N$, of small diameters $O(r_j)$. A conformal mapping framework is set up to study this free boundary problem with $\Omega_j$ being part of unknowns. For any given vorticities $\mu_1, \ldots, \mu_N$ and small $r_1, \ldots, r_N\in \BFR^+$, through a perturbation approach, we obtain such piecewise constant steady vortex patches as well as piecewise smooth Lipschitz steady vorticities, both concentrated near non-degenerate critical configurations of the Kirchhoff-Routh Hamiltonian function. When vortex patch evolution is considered as the boundary dynamics of $\p \Omega_j$, through an invariant subspace decomposition, it is also proved that the spectral/linear stability of such steady vortex patches is largely determined by that of the $2N$-dimensional linearized point vortex dynamics, while the motion is highly oscillatory in the $2N$-codim directions corresponding to the vortical domain shapes. 
\end{abstract}

\maketitle

\section{Introduction} \label{S:Intro}

Consider the incompressible Euler equation with the slip boundary condition in a bounded domain $\Omega \subset \BFR^2$ 
\bse \label{E:Euler} 
\be \label{E:Euler-E}
\p_t u + (u \cdot \nabla) u = - \nabla p  
\ee
\be \label{E:Euler-BC}
\nabla \cdot u = 0, \quad x \in \Omega  \quad \text{ and } \quad u \cdot \mathbf{N} = 0, \quad x \in \p \Omega,
\ee
\ese
where $u= (u^1, u^2)^T$ is the velocity field, $p$ is the pressure, $\mathbf{N}$ is the outward unit normal to $\p \Omega$. 
Throughout this paper, we suppose that the boundary  
\[
\p \Omega = \cup_{j=0}^n C_j, \; n\ge 0, \text{ with outward unit normal vector } \BFN|_{C_j}= \BFN_j
\] 
where each connected component $C_j$ is a sufficiently smooth simply closed curve. In particular, $C_0$ denotes the exterior boundary. 

We shall mainly work with the vorticity formulation 
\be \label{E:Euler-V}
\p_t \omega + (u \cdot \nabla) \omega =0
\ee
of the Euler equation with the vorticity given by 
\[
\omega = \nabla \times u = \p_{x_1} u^2 - \p_{x_2} u^1. 
\]
To completely determine $u$ from $\omega$, one needs to specify the circulations 
\be \label{E:circulation}
\CC_j = \oint_{C_j} u \cdot d \vec{S}, \qquad j=0, \ldots, n,
\ee
on each $C_j$, which must satisfy the condition 
\be \label{E:circ-Stokes}
\sum_{j=0}^n \CC_j = \int_{\Omega} \omega d\mu  
\ee
due to the Stokes Theorem. It is standard that all the circulations $\CC_j$ are conserved physical quantities of solutions to the Euler equation. Therefore we fix the circulations  
\be \label{E:circ-P}
\vec{\CC} = (\CC_1, \ldots, \CC_n)^T
\ee
as free parameters of \eqref{E:Euler-V}. It is well-known that $u$ is uniquely determined by $\vec{\CC}$ and $\omega$, so \eqref{E:Euler-V} is a closed PDE system which is globally well-posed for $\omega \in H^s(\Omega)$, $s > \frac 12$. Moreover, the vortex patch dynamics of \eqref{E:Euler-V} is also well-posed, namely, solutions with piecewise constant initial $\omega$ remain piecewise constant \cite{Ch93, BC93, Se94}. See Section \ref{S:Pre} for more information.  

In this paper, we shall focus on steady solutions of \eqref{E:Euler-V} in the form of 
\be \label{E:localized-V}
\omega = \sum_{j=1}^N \omega_j, \quad supp(\omega_j) = \Omega_j \subset\subset \Omega \text{ diffeomorphic to disk}, \quad \Omega_j \cap \Omega_k =\emptyset, \; \forall j\ne k. 
\ee
In particular, each vortical component $\omega_j$ has a prescribed total vorticity on the small vortical domain $\Omega_j$, namely, 
\be \label{E:localization} 
\int_{\Omega_j} \omega_j d\mu = \mu_j \ne 0, \quad diam(\Omega_j) \leq 3r_j, \quad  \int_{\Omega_j} d\mu = \pi r_j^2, \quad 0<r_j<<1. 
\ee
Here $\mu_j\ne 0$ and small $r_j>0$ are prescribed parameters as they are physical quantities conserved by the volume preserving fluid flow of the Euler equation. We shall consider the existence of both continuous steady $\omega$ and steady vortex patches, the properties of their $\Omega_j$, and the spectral analysis of the linearized vortex patch dynamics at those steady ones. 

For such concentrated vorticities, among their main characterizations are the locations and the shapes of the vortical domains $\Omega_j$. Microscopically, the vorticities $\omega_j$ are highly concentrated -- with large average due to $\mu_j \ne 0$ and away from each others and boundary. Therefore, after some appropriate rescaling, each $\omega_j$ for a steady $\omega$ should approximately be a steady vorticity on $\BFR^2$. An obvious choice of such steady vorticity in $\BFR^2$ is a radially symmetric one on disks, some of which are stable as proved in \cite{WP85}. In our construction of steady concentrated vorticities in both continuous and piecewise constant cases, each vortical domain $\Omega_j$ would be a disk of radius $r_j$ with $O(|\vec{r}| r_j^2)$ perturbations where 
\[
\vec{r} =(r_1, \ldots, r_N)^T \in (\BFR^+)^N. 
\]
Macroscopically,  as $|\vr| \to 0+$, each vortex component $\omega_j$ converges to a delta mass $\mu_j \delta_{x_j^*} (x)$ at a location $x_j^*$. Therefore $\{x_1^*, \ldots, x_n^*\}$ is expected to be a steady configuration of the point vortex system, which is a $2N$-dim Hamiltonian ODE system 
\be \label{E:PVD}
\p_t X = \Lambda_N^{-1} J_N \nabla H_{\vec{\CC}} (X), \quad X \in \Omega^N, 
\ee
on $\Omega^N$ with the Kirchhoff-Routh Hamiltonian 
$H_{\vec {\CC}}$ given in \eqref{E:E0} and symplectic operator $\Lambda^{-1} J_N$ given in \eqref{E:symplectic-1}.  
When $\vec{\CC}=0$, such system was first introduced by Helmholtz \cite{He1858, Lin43} and proved rigorously in \cite{MP83, Tu87, MP93}. See also \cite{Se98, CW18b, DDMW18} and, for a related problem of the limit motion of concentrated vorticities in background vorticity distributions, see \cite{MP91, LM09, BJ11, CW18a}. As the first step to understand the dynamics of the concentrated vorticities, through a perturbation approach, we shall study  steady concentrated vorticities, both Lipschitz and piecewise constant on $\Omega$, located near a non-degenerate critical configuration of $H_{\vec{\CC}}$ with each vortical domain being an $O(|\vec{r}|r_j^2)$ perturbation to the disk $B_{r_j} (x_j^*)$ of radius $r_j$. 
 As discussed in Subsection \ref{SS:Ham-VP}, the vortex patches evolve as vortex patches \cite{MB02}.  Treating it as an evolution problem of the vortical boundary curves $\p \Omega_j$,  we also study the spectral stability of these steady vortex patches. The framework based on conformal mappings set up in this paper will also be the one in our forthcoming work on the nonlinear local dynamics near steady vortex patches. 

The first theorem is on the existence and local uniqueness of steady concentrated vortex patches which is proved via an Implicit Function Theorem argument. 

\begin{theorem} \label{T:VP}
Given circulations $\vec{\CC} \in \BFR^n$, $s>\frac 32$, integer $N>0$, vorticities $(\mu_1, \ldots, \mu_N) \in (\BFR\backslash \{0\})^N$, and a non-degenerate critical point $X_*= (x_1^*, \ldots, x_N^*) \in \Omega^N$ of $H_{\vec{\CC}}$ (defined in \eqref{E:E0}), there exists $\ep_0>0$ such that for any $\vec{r}=(r_1, \ldots, r_N)\in (\BFR^+)^N$ with $|\vec{r}| < \ep_0$, there exists a steady piecewise constant vortex patch $\omega$ satisfying \eqref{E:localized-V} and \eqref{E:localization}. Moreover for each vortical domain $\Omega_j$, $\p (r_j^{-1} \Omega_j)$ is an $O(|\vec{r}|r_j)$ perturbation to the circle $S^1$ in the $H^s (S^1)$ topology and the steady $\omega$ is locally unique in this class and independent of $s$. 
\end{theorem}

\begin{remark}
In particular, it is not assumed that the sizes of the small $r_j$ of each vortical domain are of the same order. See Remark \ref{R:shape} for details on the shape of $\Omega_j$, which is like an ellipse in the leading order. 
\end{remark}

The next theorem proved in Section \ref{S:stability} states that spectral properties of the linearized vortex patch dynamics at such steady concentrated vortex patches found in Theorem \ref{T:VP} are largely determined by those of the point vortex dynamics \eqref{E:PVD} linearized at $X_*$. The proof is based on a perturbation argument where the Hamiltonian structure of the problem is used. More concretely, we represent the vortical domains by {\it conformal mappings} defined on the unit disk $B_1$. Let $\mathcal{X}^s$, $s\ge 0$, denote the space of  $H^s(S^1)$ variations  of the boundaries of vortical domains modulo certain symmetry (essentially due to the Mobius group on $B_1$). We denote by $A(\vr): \mathcal{X}^s \supset \mathcal{X}^{s+1} \to \mathcal{X}^s$ the linearized {\it boundary} dynamics of vortex patches at the concentrated steady vortex patch found in Theorem \ref{T:VP}. Consider the spectrum of $A(\vr)$ and the flow  $e^{tA(\vr)}$ of the linearized vortex patch dynamics \eqref{E:LVP} acting on $\mathcal{X}^s$. 

\begin{theorem} \label{T:spectrum} 
The flow $e^{tA(\vr)}$ of the   linearized vortex patch dynamics \eqref{E:LVP} is well-posed. The spectrum $\sigma\big(A(\vr)\big)$, consisting of only isolated eigenvalues, is independent of $s\ge 0$ and there exists a $2N$-dim subspace $Z_0\subset \cap_{s\ge 0} \mathcal{X}^s$ such that the following hold. 
\begin{enumerate}
\item For each $s\ge 0$, there exists a closed subspaces $Z_{Y^s} \subset \mathcal{X}^s$ such that 
\[
Z_{Y^{s'}} \subset Z_{Y^s}, \;\, \forall \ s' >s, \quad \mathcal{X}^s = Z_0 \oplus Z_{Y^s}, \quad A(\vr) Z_0 = Z_0, \quad A(\vr) Z_{Y^{s+1}} = Z_{Y^s}
\]
and $Z_0$ and $Z_{Y^s}$ are approximately orthogonal in $L^2$ and $H^s(S^1)$ metric. 
\item There exists $C>0$ such that 
\[
|\lambda| \le C, \;\; \forall \ \lambda \in \sigma \big(A (\vr) |_{Z_0} \big), \quad  \sigma (A|_{Z_{Y^s}}) \subset \{ \lambda \in i\BFR\mid |\lambda| \ge C^{-1} |\vr|^{-2}\}.
\]
\item In an appropriate coordinate system and for some $C>0$
\[
|\big(A(\vr)|_{Z_0}\big) - \Lambda_N^{-1} J_N D^2 H_{\vec{\CC}} (X_*)| \le C|\vr|. 
\]
\item The linear equation $\p_t v = A(\vr) v$ has a Hamiltonian structure on $\mathcal{X}^0$ whose Hamiltonian is given by a bounded symmetric linear operator $\BFL(\vr)$ on $\mathcal{X}^0$ such that 
\[
\langle \BFL(\vr) e^{tA(\vr)}v_1, e^{tA(\vr)} v_2 \rangle_{L^2} = \langle \BFL(\vr)v_1, v_2 \rangle_{L^2}, \quad \langle \BFL(\vr)v, v \rangle_{L^2} \ge C^{-1}  |\vr|^{-2} |v|_{\mathcal{X}^0}^2, \;\; \forall \ v \in Z_{Y^0}. 
\]
\end{enumerate}
\end{theorem}

The above last statement implies the linear stability of $e^{tA(\vr)}$ on $Z_{Y^0}$, namely,  
\[
\exists C>0, \; \text{ such that } \; |e^{tA}|_{Z_{Y^0}}| \le C, \quad \forall \  t\in \BFR.
\]
Therefore the stability of the linearized vortex patch dynamics $e^{tA(\vr)}$ is determined by the $A(\vr)|_{Z_0}$ which is a Hamiltonian perturbation to the linearized point vortex dynamics. Due to its Hamiltonian structure, the following stability statements hold. More can be found in Remark \ref{R:stability} in Section \ref{S:stability}. Readers are also referred to, for example, \cite{Ar89, LZ17} for more perturbation results on general Hamiltonian operators.

\begin{corollary}
The following statements hold. 
\begin{enumerate}
\item If $\pm D^2 H_{\vec{\CC}} (X_*) >0$, then $e^{tA(\vr)}$ is stable.
\item If $\Lambda_N^{-1} J_N D^2 H_{\vec{\CC}} (X_*)$ has an eigenvalue $\lambda \notin i\BFR$, then $e^{tA(\vr)}$ is unstable and has exponential trichotomy. 
\end{enumerate}
\end{corollary} 

One is reminded that the above stability analysis is for the boundary evolution of the vortical domains. The subspaces $Z_0$ and $Z_{Y^s}$ roughly represent the variations of the vorticity locations and shapes of the vortical domains, respectively. We also see the separation of scales in the $Z_0$ and $Z_{Y^0}$ directions, where the motion is much faster and very oscillatory along $Z_{Y^0}$. Therefore the vortex patch dynamics near such steady concentrated patch is roughly a normally elliptic type singular perturbation problem. The vortex locations approximately follow the point vortex dynamics and the shapes of the near circular vortical domains evolve in a fast rotating way with large angular velocity. Theorem \ref{T:spectrum} lays the cornerstone for future studies of local dynamics, such as invariant manifolds and special solutions, near steady concentrated vortex patches. 

We also prove the existence of Lipschitz steady concentrated vorticities in Section \ref{S:C0}. Again it is not assumed that the sizes of the small $r_j$ of each vortical domain are of the same order. Remark \ref{R:shape} still applies in the case. Though also obtained via an Implicit Function Theorem argument, the lack of local uniqueness of such steady concentrated vorticities in the following theorem is due to the infinitely many choices of the dependence of steady vorticities on  the corresponding stream functions. 

\begin{theorem} \label{T:C0}
Given circulations $\vec{\CC} \in \BFR^n$, integer $N>0$, vorticities $(\mu_1, \ldots, \mu_N) \in (\BFR\backslash \{0\})^N$, and a non-degenerate critical point $X_*= (x_1^*, \ldots, x_N^*) \in \Omega^N$ of $H_{\vec{\CC}}$ (defined in \eqref{E:E0}), there exists $\ep_0>0$ such that for each $\vec{r}=(r_1, \ldots, r_N)\in (\BFR^+)^N$ with $|\vec{r}| < \ep_0$, there exists a steady concentrated vorticity $\omega \in C^{0, 1} (\Omega)$ satisfying \eqref{E:localized-V} and \eqref{E:localization}. Moreover, for any $s>\frac 32$ such that $\p \Omega$ consists of $H^{s}$ curves, there exists $\ep_s>0$ such that for $|\vr|<\ep_s$, $\omega$ is in $H^{s+\frac 12}$ on each of $\bar \Omega_1, \ldots, \bar \Omega_N, \bar \Omega \backslash \cup_{j=1}^N \Omega_j$ and for each vortical domain $\Omega_j$, $\p (r_j^{-1} \Omega_j)$ is an $O(|\vec{r}|r_j)$ perturbation to the circle $S^1$ in the $H^s (S^1)$ topology.
\end{theorem}
 
The stability of such $C^{0,1}$ concentrated steady vorticity is a much more subtle issue as it involves the stability of the vorticity profiles themselves and their interactions with the shapes of the vortical domains. We leave it for future considerations. 

There are rich literatures on the existence of steady (or relatively steady) concentrated vorticities as well as their stability and even some parts of our above results had been derived previously. In \cite{Wan88}, Wan obtained steady concentrated vortex patches near  non-degenerate steady point vortex configurations on general 2-dim smooth bounded domains as well as near rotating ones on $\BFR^2$. More recently, on a simply connected domain $\Omega$, by carefully studying a delicate semilinear elliptic equation on $\Omega$ with a singular parameter, based on either a variational or Lyapunov-Schmidt reduction approach,  steady concentrated $C^{0,1}$ vorticities \cite{SV10, CLW14} and vortex patches \cite{CPY15} are also derived near non-degenerate steady point vorticities. While employing subtle and deep analysis from semilinear elliptic problems and calculus of variations, only the upper (and sometimes lower) bounds of the vortical domains $\Omega_j$ were given in \cite{SV10, CLW14, CPY15}, however, their smoothness and asymptotic properties were not present which would be crucial for the linearized analysis. More references include \cite{Tu83, GIPZ10, CGPY17} {\it etc.} The sizes $r_j$ of each $\Omega_j$ are all assumed to be of the same order in these references. The argument in \cite{Wan88} based on \cite{WP85} is more perturbational and directly on the unknown $\p \Omega_j$ written as graphs over $S^1$ in polar coordinates. While our proof of Theorem \ref{T:VP} is somewhat parallel to that in \cite{Wan88}, we adopted a quite different conformal mapping parameterizations of the vortical domains, very carefully tracked the smoothness of the mappings in function spaces of $\p \Omega_j$ and their asymptotic properties with respect to $\vr=(r_1, \ldots, r_N)$ which may not be of the same order of smallness, and avoided the reduction argument by a simple rescaling. Our approach is easily adapted to also yield the above $C^{0,1}$ steady concentrated vorticities (Theorem \ref{T:C0}). Many of the previous stability results of steady vortex patches are nonlinear with respect to the $L^p$ metric for the vorticity, following a Lyapunov function approach by Arnold, see \cite{Tu83, WP85, Wan86, MP94, CW17a}, {\it etc}. While being nonlinear, the $L^p$ control on the vorticity does not yield as much information on the geometric evolution of the vortical domain as those results in smooth function spaces of the vortical domain boundaries. The studies on the linearization in the latter framework include \cite{WP85, Wan86, Wan88} {\it etc.} on circular, elliptic, some rotating vorticities {\it etc.} in $\BFR^2$. While invariant subspace splitting was also obtained in \cite{Wan86, Wan88}, the construction in Section \ref{S:stability} is more explicit with detailed estimates with respect to $\vr=(r_1, \ldots, r_N)$. See also \cite{EP13} for a linearization framework based on contour integrals. 

This paper is the first step of our longer term project of studying the temporal Euler dynamics of the vortical domain boundaries of concentrated and compactly supported vorticities. The analysis on compactly supported vorticities is actually a free boundary problem as the vortical domains are largely the main unknowns. In fact the main goal of this paper is to establish a mathematically operable rigorous analytic framework to study the dynamics of  the rather geometric objects of vortical domain boundaries. 
Due to the Poisson equations involved in the relationship between the vorticity and the stream function, we parametrize a class of vortical domains by unique conformal mappings  which interact well with $\Delta$. See some basic properties of conformal mappings in Subsections \ref{SS:CM} and \ref{SS:Ham-VP}, which are improved from those in \cite{SWZ13}. We also would like to point out \cite{Bu82, HMV13} for conformal mappings used to study rotating vortex patches in $\BFR^2$. The construction of steady states and the analysis of their linear stability are merely our initial attempts to study in this framework the local dynamics, such as invariant manifolds, of concentrated vorticity as a singular perturbation problem involving multiple spatial and temporal scales. 

The paper is organized as follows. In Section \ref{S:Pre}, we outline some background materials including the stream functions, the Hamiltonian structures for the  2-dim Euler equation in the general vorticity formulation and for vortex patches dynamics, the conditions on steady concentrated vorticities, and the conformal mapping parametrization of domain. Section \ref{S:VP} is devoted to the construction of steady concentrated vortex patches, whose spectral stability is analyzed in Section \ref{S:stability}. Finally, $C^{0,1}$ steady concentrated vorticities are obtained in Section \ref{S:C0}.

\vspace{.08in} 
\noindent {\bf Notations.} 
\begin{itemize} 
\item Throughout the paper, $C>0$ denotes a generic upper bound which may change from line to line, but independent of $\vr$. 
\item We usually use $\omega$, $\wt \omega_j$, etc. to denote quantities related to vorticities, 
$\Omega_j$
vortical domains, $\Gamma$, $\Gamma_j$, etc. conformal mappings defined on unit disk $B_1$ where $B_R = \{z \in \BFC^1\mid |z|<R\}$, $J$, $\mathbf{J}$, etc. various operator satisfying $J^* = -J$, and $\Psi$ the stream function such that velocity$=J \nabla \Psi$, where specifically $J =  \begin{pmatrix} 0 & -1 \\ 1 & 0 \end{pmatrix}$. 
\item We use $d\mu$ to denote the Lebesgue measure or $d\mu_x$ the  measure in $x$ variable.  
\item Along $\p \Omega_j$ of each vortical domain $\Omega_j$, $j=1, \ldots, N$,  and each component of $\p \Omega$, $N_j$, $T_j = J N_j$, $\BFN$, and $\BFT = J\BFN$ denote the {\it outward} unit normal vector 
and the 
unit tangent vector. 
\item The operator $\Delta_0^{-1}$, defined in Lemma \ref{L:Stream-F}, denotes the inverse Laplacian with circulation 0 along each $C_j$, $1\le j\le n$. 
\item Usually $D$, $\p$, $'$, and $\nabla$ are used for differentiations in finite dimensional spaces, while $\CD$ is reserved for differentiations in infinite dimensional function spaces. 
\item We always use $\langle \cdot, \cdot \rangle$ to denote the $L^2$ (or some times $\BFR^k$) duality pair. 
Sometimes complex numbers are treated as 2-dim vectors and notation $z_1 \cdot z_2 =$Re$(\bar z_2 z_1)$ is used. 
\end{itemize}

\section{Preliminaries: Stream functions, Hamiltonian structures, point vortex dynamics, and parametrization of domains} \label{S:Pre}

In this section we discuss some basic issues related to the study of concentrated steady vorticities of the Euler equation in $\Omega$.

\subsection{Stream functions} \label{SS:Stream-F}

Stream functions and vorticty are fundamental quantities for the Euler equation in 2-dim. For any vector field $u$ on $\Omega$ satisfying \eqref{E:Euler-BC}, $\nabla \cdot u =0$ implies that locally there exists a unique (up to a constant) scalar valued function $\Psi$  such that 
\[
u = J \nabla \Psi, \quad J =  \begin{pmatrix} 0 & -1 \\ 1 & 0 \end{pmatrix}, \text{ and } \Delta \Psi = \omega.
\]
To see that this stream function $\Psi$ can actually be extended globally in $\Omega$, for any piecewise $C^1$ simply closed curve $\gamma \subset \Omega$, let $C_{j_l}$, $1\le j_l \le n$ for $l=1, \ldots, m$, be the interior components of $\p \Omega$ enclosed by $\gamma$ and then we can compute using \eqref{E:Euler-BC} 
\[
- \oint_\gamma J u  \cdot d\vec{S} = \sum_{l=1}^m  \int_{C_{j_l}} u \cdot \BFN_{j_l} dS =0. 
\]
Therefore, even though $\Omega$ may not be simply connected, such velocity field always has a stream function $\Psi$ in $\Omega$ due to \eqref{E:Euler-BC} which also implies 
\be \label{E:Stream-F-0}
\Psi|_{C_j} = const, \quad \int_{C_j} \BFN_j \cdot \nabla \Psi dS = \oint_{C_j} u \cdot d\vec{S} = \CC_j, \qquad 1\leq j \leq n.
\ee
We normalize it so that 
\be \label{E:Stream-F}
u = J \nabla \Psi, \quad \Psi|_{C_0} =0. 
\ee
The existence of such team function is standard \cite{Lin43}. In the rest of the subsection, we will provide the explicit representation of $\Psi$ in terms of $\omega$ and $\vec{\CC}$, which is a slightly different treatment from \cite{Fl99}. 

As a preparation, for any $0\leq j \leq n$, let $\CH_j(x)$ be the solution to  
\[
\Delta \CH_j =0, \text{ in } \Omega, \quad \CH_j |_{C_k} = \delta_{jk}, 
\]
where $\delta_{jk}$ is the Kronecker symbol. Clearly $\Sigma_{j=0}^n \CH_j (x)$ is harmonic and equal to $1$ on $\p \Omega$ and thus identically $1$ on $\Omega$. Therefore we skip $j=0$ and define the $n\times n$ matrix $\CN = (\CN_{jk})$, which depends only on $\Omega$, as 
\be \label{E:DN}
\CN_{jk} = \int_{C_j} \BFN_j \cdot \nabla \CH_k dS.
\ee

\begin{lemma} \label{L:DN}
It holds $\CN^T=\CN$ and $\CN >0$. 
\end{lemma}

\begin{proof} 
The symmetry of $\CN$ is due to 
\be \label{E:DN-1}
\CN_{jk} = \int_{C_j} \BFN_j \cdot \nabla \CH_k dS= \int_{\p \Omega} \CH_j \BFN \cdot \nabla \CH_k dS = \int_{\Omega} \nabla \CH_j \cdot \nabla \CH_k d\mu. 
\ee
Moreover $\forall b= (b^1, \ldots, b^n)^T \in \BFR^n \backslash \{0\}$, the above equality implies 
\[
b^T \CN b = \int_\Omega |\nabla \phi|^2 d\mu  >0, \text{ where }  \phi = \Sigma_{j=1}^n b^j \CH_j. 
\]
Here we used the fact that $\nabla \phi$ does not vanish entirely as $\phi|_{C_0} \equiv 0$ and $\phi|_{C_{j_0}} \ne 0$ for some $j_0$. Therefore the positivity of $\CN$ follows immediately. 
\end{proof}

Let $\BFG(x, y)$, $x, y \in \Omega$, $x\ne y$, be the Green function on $\Omega$ satisfying 
\[
-\Delta_y \BFG(x, y) = \delta_x(y), \;x, y \in \Omega, \; \text{ and } \; \BFG(x, y) =0, \; x \in \Omega, \; y\in \p \Omega
\]
where $\delta_x(\cdot)$ is the delta mass at $x$. 
It is standard that 
\be \label{E:GF}
\BFG(x, y) = - \frac 1{2\pi} \log |x-y| + \tilde g(x, y), \qquad \CH_j (x) = - \int_{C_j} \BFN_j (y) \cdot \nabla_2 \BFG(x, y) dS_y
\ee 
and $\tilde g(x, y) \in C^\infty (\Omega \times \Omega)$ satisfy $\tilde g(x, y) = \tilde g(y, x)$. Let 
\be \label{E:Circ-coe}
\CN^{-1} = (\CN^{jk}),  \quad \vec{c} = (c_1, \ldots, c_n)^T= \CN^{-1} \vec{\CC}, \; \; c_j = \sum_{k=1}^n \CN^{jk} \CC_k 
\ee
and 
\be \label{E:Gc}
\BFG_0 (x, y) = \BFG(x, y) + \sum_{j, k=1}^n \CN^{jk} \CH_j(x) \CH_k(y). 
\ee

\begin{lemma} \label{L:Stream-F}
Given any $\vec{\CC} \in \BFR^n$ and $\omega \in H^{s} (\Omega)$, $s> -\frac 12$, let 
\[
\Psi (x) = (\Delta_0^{-1} \omega)(x) + \sum_{j=1}^n c_j \CH_j(x), \; \text{ where }  \;(\Delta_0^{-1} \omega)(x) \triangleq \int_\Omega- \BFG_0 (x, y) \omega(y) d\mu_y . 
\]
Then $u =J \nabla \Psi$ satisfies \eqref{E:Euler-BC}, \eqref{E:circulation}, \eqref{E:circ-Stokes}, and $\nabla \times u = \omega$. Moreover, 
\[
\Psi|_{C_0} =0, \quad \Psi|_{C_j}  = c_j - \sum_{k=1}^n \CN^{jk} \int_\Omega \CH_k \omega d\mu = \sum_{k=1}^n \CN^{jk} (\CC_k -  \int_\Omega \CH_k \omega d\mu)
\]
for $1\leq j \leq n$, and 
\[
|\Psi|_{H^{s+2}} + |u|_{H^{s+1}} \le C (|\omega|_{H^s} + |\CC|). 
\]
for some $C>0$ depending only on $\Omega$. The lemma holds for $s\ge -1$ if $n=0$. 
\end{lemma}

\begin{proof}
The definition of $\Psi$, $\BFG$, and $\CH_j$ immediately imply $\Delta \Psi =\omega$ and the desired constant values of $\Psi|_{C_j}$, $0\le j\le n$. 
Clearly $\nabla \cdot u = \nabla \cdot (J\nabla \Psi) =0$. We only need to verify the circulation condition using \eqref{E:GF}, \eqref{E:Circ-coe} and the symmetry of $\BFG$
\begin{align*}
&\int_{C_j} \BFN_j \cdot \nabla \Psi dS = \sum_{k=1}^n c_k \CN_{jk} - \int_{C_j} \int_\Omega \BFN_j (x) \cdot \nabla_1 \BFG(x, y)  \omega (y) d\mu_y dS_x \\
&\qquad \qquad \qquad \qquad \qquad \qquad  - \sum_{k, l =1}^n \CN^{kl} \int_\Omega \CH_l \omega d\mu \int_{C_j} \BFN_j \cdot \nabla \CH_k dS \\
=& \CC_j + \int_\Omega\CH_j  \omega d\mu - \sum_{k, l =1}^n \CN^{kl} \CN_{jk} \int_\Omega \CH_l \omega d\mu = \CC_j. 
\end{align*}
Finally the estimates on $\Psi$ and $u$ follow from the standard elliptic estimates and the proof is complete. 
\end{proof}

\begin{remark} \label{R:Delta_0}
It is clear that $\Psi_0 = \Delta_0^{-1} \omega$ corresponds to the unique stream function satisfying $\Psi_0|_{C_0} =0$, $\Psi_0|_{C_j}=const$ for $1 \leq j \leq n$, and that its velocity field $J\nabla \Psi_0$ has vorticity $\omega$ and circulation 0 along $C_j$, $1\le j\le n$. $\BFG_0$ is the corresponding Green function. Apparently, $\Delta_0^{-1}$ is symmetric with respect to the $L^2$ inner product. If $\Omega$ is simply connected, namely $n=0$, then $\Delta_0^{-1} = \Delta^{-1}$, the inverse Laplacian with zero Dirichlet boundary condition. In this case, we may take $s\geq -1$. 
\end{remark}

\subsection{Hamiltonian structures of the Euler equation} \label{SS:Ham-Euler}

The Euler equation has the conserved energy functional 
\be \label{E:energy}
\BFE = \frac 12 \int_\Omega |u|^2 d\mu = \frac 12 \int_\Omega |\nabla \Psi|^2 d\mu. 
\ee
Writing the energy in terms of the vorticity along with an appropriate symplectic structure, the Euler equation \eqref{E:Euler-V} can be put in a Hamiltonian formulation. 

\begin{lemma} \label{L:Ham-E}
For smooth solutions, \eqref{E:Euler-V} is equivalent to
\[
\p_t \omega = J_0(\omega) \CD E_0(\omega) 
\] 
where the explicit form of $E_0(\omega)=\BFE$ is given in \eqref{E:energy-V} and  
\[
J_0(\omega) f  = \nabla \cdot \big( (u \cdot \nabla) \nabla f\big), 
\]
with $u = J\nabla \Psi$ determined by $\omega$ through Lemma \ref{L:Stream-F}. 
\end{lemma}

From \eqref{E:energy-V} and Lemma \ref{L:Stream-F}, $E_0(\omega)$ is a quadratic polynomial of $\omega \in H^s$, $s> -\frac 12$, or $s\ge -1$ if $n=0$. Moreover, $J_0(\omega)$ is skew-symmetric, namely, for smooth $\omega$, $f_1$, and $f_2$, 
\[
\langle f_1, J_0(\omega) f_2\rangle + \langle f_2, J_0(\omega) f_1\rangle =0.
\]

\begin{proof}
Let $\Psi_0 = \Delta_0^{-1} \omega$ and $\Psi_1= \Psi - \Psi_0$. One the one hand, for any $\phi$ satisfying $\phi|_{C_j} =const$ for $0 \leq j \leq n$, from Remark \ref{R:Delta_0} one can compute 
\[
\int_{\Omega} \nabla \Psi_0 \cdot \nabla \phi d\mu = \int_{\p \Omega} \phi \BFN \cdot \nabla \Psi_0  dS - \int_{\Omega}\phi \Delta \Psi_0 d\mu = - \int_{\Omega}\phi \omega d\mu,
\]
which implies 
\[
\int_{\Omega} | \nabla \Psi_0|^2 d\mu = - \int_{\Omega}\Psi_0 \omega d\mu = - \int_{\Omega}\omega \Delta_0^{-1} \omega d\mu
\]
and 
\[
\int_{\Omega} \nabla \Psi_0 \cdot \nabla \Psi_1 d\mu = - \int_{\Omega} \omega \Psi_1 d\mu = - \sum_{j=1}^n c_j \int_\Omega \CH_j \omega d\mu.
\]
On the other hand, from \eqref{E:DN-1} and \eqref{E:Circ-coe}, 
\[
\int_{\Omega} | \nabla \Psi_1|^2 d\mu = \sum_{j, k=1}^n c_j c_k \CN_{jk} = \vec{\CC}^T \CN^{-1} \vec{\CC}.
\]
Therefore we obtain 
\be \label{E:energy-V}
E_0 (\omega)  =  - \frac 12 \int_{\Omega}\omega \Delta_0^{-1} \omega d\mu - \sum_{j=1}^n c_j \int_\Omega \CH_j \omega d\mu 
+  \frac 12  \vec{\CC}^T \CN^{-1} \vec{\CC}
\ee
and 
\[
\CD E_0 (\omega) = -\Delta_0^{-1} \omega - \sum_{j=1}^n c_j \CH_j = -\Psi. 
\]
Finally 
\begin{align*} 
J_0(\omega) \CD E_0(\omega) = - \nabla \cdot \big( (u \cdot \nabla) \nabla \Psi \big) =  \nabla \cdot \big( J (u \cdot \nabla) u \big) = - \nabla \times \big( (u \cdot \nabla) u \big) = -  (u \cdot \nabla) \omega.
\end{align*} 
The proof of the lemma is complete. 
\end{proof}

The operator $J_0(\omega)$ may not be invertible. For example, if $\omega$ is radial on $\Omega= B_1$, then $u$ is angular and $J_0(\omega) f =0$ for any radial $f$. Therefore this $J_0(\omega)$ does not immediately induces a symplectic inner product, which usually is associated with  $J_0(\omega)^{-1}$. For more results on the Hamiltonian structures of the Euler equation in Eulerian coordinates, see, for example, \cite{Ol82}.

\subsection{Concentrated vorticies and point vortex dynamics} \label{S:PVD}

Fix $(\mu_1, \ldots, \mu_N) \in (\BFR\backslash \{0\})^N$ and consider a family of concentrated vorticities $\omega = \Sigma_{j=1}^N \omega_j$, parametrized by $\vec{r} = (r_1, \ldots, r_N)^T \in (\BFR^+)^N$, satisfying \eqref{E:localized-V}, \eqref{E:localization},  and 
\be \label{E:t-omega}
\int_{\BFR^2} \wt \omega d\mu =1, \qquad \omega_j( x) -  r_j^{-2}\mu_j \wt \omega_j\big( r_j^{-1} (x-x_j)\big) \to 0, \; \text{ as } |\vec{r}|\to 0,
\ee 
for some $\wt \omega_j(x)$ in some appropriate sense. From \eqref{E:energy-V}, \eqref{E:GF}, \eqref{E:Gc}, and Lemma \ref{L:Stream-F}, the energy takes the form 
\begin{align*}
& E_0(\omega) - \frac 12  \vec{\CC}^T \CN^{-1} \vec{\CC} + \sum_{j=1}^n c_j \int_\Omega \CH_j \omega d\mu = \frac 12 \int_{\Omega \times \Omega} \BFG_0( x, y) \omega (x)  \omega(y) d\mu_{x,y}\\
=&-\frac 1{4\pi} \sum_{j=1}^N \int_{\Omega_j \times \Omega_j} \log|x-y| \omega_j (x) \omega_j(y) d\mu_{x,y} + 
\frac 12 \sum_{ 1\leq j \ne k \leq N} \int_{\Omega_j \times \Omega_k} \BFG_0(x, y)  \omega_j (x) \omega_k(y) d\mu_{x,y} \\
&+\frac 12 \sum_{j=1}^N \int_{\Omega_j \times \Omega_j}  \omega_j (x) \omega_j(y) \big(\tilde g(x,y) + \sum_{l, m=1}^n \CN^{l, m} \CH_l(x) \CH_m(y) \big) d\mu_{x,y}.
\end{align*}
As $|\vec{r}| \to 0$, each $\omega_j$ converges to the delta mass $\mu_j \delta_{x_j} (x)$ at $x_j$ and thus 
\be \label{E:E0}
E_0(\omega) - \frac 12  \vec{\CC}^T \CN^{-1} \vec{\CC} - \sum_{j=1}^N   E_{\BFR^2} (\omega_j)  \longrightarrow H_{\vec{\CC}} (X), \quad X = (x_1, \ldots, x_N)^T,
\ee
where 
\begin{align*}
E_{\BFR^2} (\omega) =&- \frac 1{4\pi} \int_{\BFR^2 \times \BFR^2} \log|x-y| \omega(x)\omega(y) d\mu_{x, y} =  E_{\BFR^2} (\wt \omega_j) - \frac 1{4\pi} \mu_j^2  \log r_j + o(1),  \\
H_{\vec{\CC}} (X) = &- \sum_{l=1}^n \sum_{j=1}^N c_l \mu_j \CH_l (x_j)  + \sum_{j=1}^N \mu_j^2 g(x_j)  + \frac 12 \sum_{ 1\leq j \ne k \leq N}  \mu_j \mu_k\BFG_0(x_j, x_k) \\
g(x) = & \frac 12\big(\tilde g(x,x) + \sum_{l, m=1}^n \CN^{l, m} \CH_l(x) \CH_m(x) \big). 
\end{align*}
Clearly $E_{\BFR^2}$ is the energy functional of the 2-dim Euler equation on $\BFR^2$ which is invariant under spatial translation and rescaling. \\

\noindent
{\it Heuristic implication on the steady problem.} Intuitively, in the limit as $|\vec{r}| \to 0+$, the profiles $\wt \omega_j$ of the vorticity and the locations $x_j$ are decoupled. In particular each $\wt \omega_j$ is governed by the energy $E_{\BFR^2}$ of the Euler equation on $\BFR^2$, while $X = (x_1, \ldots, x_N)\in \BFR^{2N}$ is by $H_{\vec{\CC}}  (X)$.  As steady states of the Euler equation corresponding to critical points of $E_0$, those with $|\vec{r}|<<1$ must satisfy $\CD E_{\BFR^2}(\wt \omega_j) \approx 0$, $1\leq j \leq N$, and $\nabla H_{\vec{\CC}} (X) \approx 0$. This leads to the perturbation type existence results in the main theorems. In particular $\wt \omega_j$ are roughly taken as stable radial distributions on disks $\Omega_j$. \\ 

\noindent
{\it Heuristic implication on the dynamic problem.} Even though we mainly focus on stationary solutions and some stability properties, it is worth pointing out a heuristic observation on the Euler dynamics for concentrated vorticities. On the one hand, for $|\vec{r}| <<1$, the vorticities $\omega_j$ barely see each other or the boundary $\p \Omega$ in the leading order. Therefore the leading order dynamics of $\omega_j$ should be the Euler equation on $\BFR^2$ (also implied by the energy $E_{\BFR^2}(\omega_j)$ in the above \eqref{E:E0}) at the speed of $O(r_j^{-2})$. This scale of evolution is due to the spatial scale of $O(r_j)$ and the scaling invariance of the Euler equation on $\BFR^2$: if $\omega(t, x)$ is a solution, so is $\ep^{-2} \omega (\ep^{-2}t, \ep^{-1} x)$. On the other hand, it has been rigorously proved \cite{MP83, Tu87, MP93} that the dynamics of the vortex locations $X = (x_1, \ldots, x_N)$ is governed by the Hamiltonian ODE system $ \p_t X = \Lambda^{-1} J_N \nabla H_{\vec{\CC}}(X)$ with the Hamiltonian $H_{\vec{\CC}} $ and the symplectic operator $\Lambda^{-1} J_N$ where 
\be \label{E:symplectic-1}
\Lambda_{2N \times 2N} = diag (\mu_1 I_{2\times2}, \ldots, \mu_N I_{2\times2}) \quad \text{ and } \quad (J_N)_{2N \times 2N} =  diag(J, \ldots, J).
\ee
Therefore the dynamics of concentrated vorticities is a PDE singular perturbation problem with two time scales of order $O(1)$ and $O(|\vec{r}|^{-2})$ and spatial scales $O(1)$ and $O(|\vec{r}|)$, which is in our long term plan. 

To end this subsection, we give the following statement on steady states. 

\begin{lemma} \label{L:steady-C0}
Suppose $\omega = \Sigma_{j=1}^N \omega_j \in C^0(\overline {\Omega})$ satisfies that 
\begin{enumerate}
\item $\Omega_j =$supp$(\omega_j)$, $j=1, \ldots, N$, are mutually disjoint and have $H^s$ boundary, $s>\frac 32$, 
\item for each $j$, $\Delta \Psi = f_j(\Psi)$ on $\Omega_j$ for some $f_j \in C^1$, where $\Psi$ is  in Lemma \ref{L:Stream-F}, 
\end{enumerate}
then $\omega$ is a steady state of the Euler equation.
\end{lemma}

Since $\omega$ may not be in $C^1(\Omega)$, here it is a steady state of \eqref{E:Euler-V} in the weak sense, while the corresponding velocity $u = J\nabla \Psi$ being steady of \eqref{E:Euler} in the classical sense. 

\begin{proof}
In $\Omega \backslash (\cup_{j=1}^N \p \Omega_j)$, either $\omega \equiv 0$ or $\omega = \Delta \Psi = f_j (\Psi)$, so the steady equation $u \cdot \nabla \omega = J \nabla \Psi \cdot \nabla \omega =0$ of \eqref{E:Euler-V} is apparently satisfied. In general, for any $\phi \in C_0^\infty (\Omega)$, we have  
\begin{align*}
& \int_\Omega \omega u \cdot \nabla \phi d\mu = \sum_{j=1}^N \int_{\Omega_j} \omega u \cdot \nabla \phi d\mu = - \sum_{j=1}^N \int_{\Omega_j} (u \cdot \nabla \omega + \omega \nabla \cdot u) \phi d\mu =0. 
\end{align*}
It implies $\nabla \cdot (\omega u) =0$ in the weak sense, which implies lemma due to $\nabla \cdot u=0$. 
\end{proof}

\subsection{Conformal mapping parametrizations} \label{SS:CM}

As the relationship between the vorticity and stream function is given by a Poisson equation, conformal mappings provide an ideal way to parametrize the unknown vortical domains $\Omega_j$ of the steady concentrated vorticities. However, since the group of conformal mappings on $B_1$ is 3-dim, given any simply connected domain, its conform mappings parametrization by $B_1$ involve 3 degree of freedom. We shall use the following lemma to fix these freedoms which also eliminates certain degeneracy in the analysis. 

\begin{lemma} \label{L:CM}
Let $k\ge 2$ be an integer and $f\in C^k(\overline{B_1}, \BFC)$  be a conformal mapping. 
Then there exist $\alpha \in S^1$ and $c \in \BFC \cap B_1$ such that $\p_z f_1(c, \alpha, 0)>0$ and $\p_z^k f_1(c, \alpha, 0) =0$, where $f_1 (c, \alpha, z) = f \big(\alpha \frac {z+c}{1+\bar c z} \big)$. 
Moreover, if 
a conformal mapping $f_0 \in C^1 (\overline{B_1})$ satisfies 
\be \label{E:CM-A}
\p_z^k f_0 (0)=0, \quad |\p_z^{k+1} f_0 ( 0)|  \ne k(k-1)|\p_z^{k-1} f_0 ( 0)|
\ee
then such $(c, \alpha)$ are unique and $C^1$ with respect to $f$ near $f_0$ in $C^1 (\overline{B_1})$. 
\end{lemma}

\begin{proof}
It is easy to make $\p_z f_1(c, \alpha, 0)>0$ by adjusting $\alpha$ and we focus on the property $\p_z^k f_1 (c, \alpha, 0)=0$ by adjusting $c\in B_1$, while we neglect the parameter $\alpha$ for simplicity of notations. One may compute 
\[
\partial_{z}^k f_1 (c, 0) =  \frac {k!}{2\pi i}\int_{S^1} z^{-k-1} f\Big(\frac {z+c}{1+\bar c z}   \Big) dz.
\] 
Let $c= (1-\delta) e^{i\theta_0}$, $\delta>0$, and parametrize $z= e^{i\theta}$ on $S^1$. We have 
\[
\partial_{z}^k f_1\big((1-\delta) e^{i\theta_0}, 0\big) = \frac {k!}{2\pi } \int_{-\pi}^\pi e^{-ik\theta} f\Big(\frac {e^{i\theta} + (1-\delta) e^{i\theta_0}}{1+ (1-\delta) e^{i(\theta -\theta_0)}}  \Big) d\theta.
\]
Let 
\[
e^{i\beta} =  \frac {e^{i\theta} + (1-\delta) e^{i\theta_0}}{1+ (1-\delta) e^{i(\theta -\theta_0)}}
\]
which is a 1-1 correspondence and 
\[
e^{i \theta} = \frac {e^{i\beta} - (1-\delta) e^{i\theta_0}}{1-(1-\delta) e^{i(\beta -\theta_0)}}, \qquad 
e^{i\theta} d \theta = \frac { \delta(2-\delta) e^{i\beta} }{(1-(1-\delta) e^{i(\beta -\theta_0)})^2} d\beta. 
\]
Therefore, 
\begin{align*}
& \partial_{z}^k f_1\big((1-\delta) e^{i\theta_0}, 0\big)  =\delta(2-\delta)\frac {k!}{2\pi } \int_{-\pi}^\pi \frac {(1-(1-\delta) e^{i(\beta -\theta_0)})^{k-1}}{(e^{i\beta} - (1-\delta) e^{i\theta_0})^{k+1}} f(e^{i\beta}) e^{i\beta} d\beta \\
=&  \delta(2-\delta) \partial_{z}^k \Big( \big(1-(1-\delta) e^{-i\theta_0} z\big)^{k-1} f(z)\Big)|_{z= (1-\delta) e^{i\theta_0} }\\
=&  \delta(2-\delta) \Sigma_{l=1}^k \frac {k! (k-1)!} {l! (l-1)! (k-l)!} \big(-(1-\delta) e^{-i\theta_0}\big)^{k-l} \Big((1-(1-\delta) e^{-i\theta_0}z)^{l-1} \p_z^l f (z)\Big)|_{z= (1-\delta) e^{i\theta_0}}\\
=& 2\delta (-1)^{k-1} k! e^{-i(k-1)\theta_0} \p_z f (e^{i\theta_0}) + O (\delta^2), \quad \text{ as } \delta \to 0+.
\end{align*}
Since $\p_z f(z) \ne 0$ on $B_1$ and thus $\p_z f(e^{i\theta_0})$, $\theta_0 \in [0, 2\pi]$, is deformable to $1$, we obtain that, as a function of $c$, the degree of $ \partial_{z}^k f_1 (c, 0)  $ is $1-k$. So it achieves zero at some $c \in B_1$. 

When $f_0 \in C^1(\overline{B_1})$, clearly $\partial_{z}^k f_1 (c, 0) $ is a $C^1$ mapping from $B_1 \times C^1(\overline{B_1})$ to $\BFR^2$. Recall $\p_z^k f_0 (0) =0$. Differentiating $\partial_{z}^k f_1 (c, 0) $ at $c=0$ and $f=f_0$ we have, for any $v \in \BFR^2 \sim \BFC$, 
\[
D_c \big(\partial_{z}^k f_1 (c, 0)\big)|_{c=0} v = \frac {k!}{2\pi i} \int_{S^1} z^{-k-1} \p_z f_0 (z) (v - z^2 \bar v)  dz=\p_z^{k+1} f_0 ( 0) v - k(k-1)\p_z^{k-1} f_0 ( 0) \bar v. 
\]
The local uniqueness of $c$ and its smooth dependence on $f \in C^1(\overline{B_1})$ follow immediately from the assumption \eqref{E:CM-A} and the Implicit Function Theorem.
\end{proof}

For this paper the most relevant part of the lemma is the second half with $f_0(z)=z$ and $k=2$, which means that any domain close to a disk can be uniquely represented by a conformal mapping $f$ with $\p_z^2 f (0)=0$ and $\p_z f(0)>0$. 

For $j=1, \ldots, N$, by Lemma \ref{L:CM}, each vortical domain $\Omega_j$ whose shape is close to $r_j  B_1$ is the image of a unique conformal mapping 
\be \label{E:CM-2-1}
\Gamma_j (z) = x_j + r_j a_1^j \big(z+ \wt \Gamma_j (z) \big), \quad \wt \Gamma_j (z) = \sum_{m=3}^\infty (a_m^j + i b_m^j) z^m.
\ee
Since these conformal mappings are harmonic in $B_1$ and thus can be determined by its boundary values, actually by only the real parts of its boundary values as the imaginary parts are Hilbert transforms of the real parts. Let 
\be \label{E:CM-2-2}
\beta_j (\theta) = \text{Re} \wt \Gamma_j (e^{i\theta}) = \sum_{m=3}^\infty (a_m^j \cos m\theta - b_m^j \sin m \theta). 
\ee
Apparently, 
for $s>0$, it holds that the above defined correspondence between $\wt \Gamma_j$ and $\beta_j$ is an isomorphism between 
\be \label{E:CM-2-3}
\{\text{holomorphic } f\in H^{s+\frac 12}(B_1, \BFC) \mid \p_z^m f(0) =0, \; m=0, 1,2\} \; \text{ and } \; X^s
\ee
where 
\be \label{E:Xs}
X^s \triangleq \{\beta = \sum_{m=3}^\infty (a_m \cos m\theta - b_m \sin m \theta) \in H^s(S^1, \BFR)
\}.
\ee
For $s> \frac 32$, let 
\be \label{E:Rs}
R_s = (3^{2-2s} + 4^{2-2s} + \ldots )^{-\frac 12},
\ee
then a simple argument based on Cauchy-Schwartz inequality implies 
\[
|\p_z \wt \Gamma_j|_{C^0(B_1)} \le (\sqrt{\pi} R_s)^{-1} |\beta_j|_{H^s (S^1)}. 
\]
Therefore if $|\beta_j|_{H^s} < R_s$, $\Gamma_j$ defined in \eqref{E:CM-2-1} and  \eqref{E:CM-2-2} is a conformal mapping. 
In order to satisfy the area requirement on $\Omega_j$ in \eqref{E:localization}, $a_1^j$ is determined by $\wt \Gamma_j$ as 
\be \label{E:a1j}
a_1^j = \big(\frac 1\pi \int_{B_1} |1 + \p_z \wt \Gamma_j|^2 d\mu\big)^{-\frac 12} = \sqrt{\pi}\big( \int_{B_1} 1 + |\p_z \wt \Gamma_j|^2 d\mu\big)^{-\frac 12},
\ee
where we also used $\int_{B_1} \p_z \wt \Gamma_j d\mu =0$. 

According to these observations and Lemma \ref{L:CM},  each $\Omega_j$ whose shape is close to $r_j B_1$ and with $H^s$ boundary is uniquely parametrized by $x_j \in \Omega$ and $\beta_j \in X^s$. Therefore collections of vortical domains $\{ \Omega_j\mid j=1, \ldots, N\}$ are identified with elements in 
\be \label{E:sigma-rho}
\Sigma_\rho \times (X^s)^N, \; \;
\Sigma_\rho \triangleq \{ X =(x_1, \ldots, x_N) \in \Omega^{2N} \mid dist(x_j, \p \Omega), \, |x_j - x_{j'}| > \rho, \, \forall j \ne j' \}
\ee
where 
$\rho> 0$ is fixed, $s>\frac 32$, $|\beta_j|_{X^s} <R_s$, and $|\vec{r}|<<\rho$.   In particular, this set also becomes an open subset of the phase space of the vortex patch dynamics. 

\subsection{Vortex patches dynamics in conformal mapping parametrization and its Hamiltonian structure} 
\label{SS:Ham-VP}

In this subsection we consider the vortex patch dynamics, where 
\be \label{E:VP-omega} 
\omega=\sum_{j=1}^N \omega_j, \quad \omega_j = \frac {\mu_j}{\pi r_j^2} \chi_{\Omega_j}, \quad |\Omega_j| = \pi r_j^2, \qquad j =1, \ldots, N,
\ee
and put it in a slightly different Hamiltonian framework. Due to the volume preserving transport of the vorticity, it is well-known that an initial collection of disjoint vortex parches $\omega = \sum_{j=1}^N \omega_j$ with reasonably smooth $\p \Omega_j$ evolves as a collection of disjoint vortex patches with $\omega (t)$ in the same form \eqref{E:VP-omega} and the same areas $|\Omega_j (t)| = \pi r_j^2$, while $\mu_j$ and $r_j$ are independent of $t$. Therefore in the vortex patch dynamics of such a $\omega(t)$, which is equivalent to a collection $\big(\p \Omega_1 (t), \ldots, \p \Omega_N(t)\big)$ of boundary curves of vortical domains, the phase space is equivalent to the $\infty$-dim `manifold'  
\[
\CM= \{(\p \Omega_1, \ldots, \p \Omega_N) \mid \Omega_j \subset \subset \Omega, \; j=1, \ldots, N, \text{ disjoint } \& \text{ `smooth', } |\Omega_j| = \pi r_j^2 \}
\]
consisting of the collections of the boundaries $\p \Omega_j$ of the vortex patches. In the dynamics of the Euler equation, given such a collection of vortex patches, the normal velocity along each $\p \Omega_j$ is given by 
\be \label{E:VP-NV}
u \cdot N_j = J \nabla \Psi \cdot N_j = - \nabla_{T_j} \Psi 
\ee
where $\Psi$ is determined by $\omega$ and Lemma \ref{L:Stream-F} and $T_j = JN_j$ is the counterclockwise unit tangent vector of $\p \Omega_j$ and thus $\nabla_{T_j}$ the derivative with respect to the arc length along $\p \Omega_j$ counterclockwisely. Therefore we have 

\begin{lemma} \label{L:steady-VP} 
A vortex patch in the form of \eqref{E:VP-omega} satisfying \eqref{E:localized-V} \eqref{E:localization} with $\p \Omega_j \in C^1$, $j=1, \ldots, N$, is steady if and only if $\Psi|_{\p \Omega_j} =const$ for all $j$.  
\end{lemma}

As in Subsection \ref{SS:CM}, we parametrize (an open subset of) $\CM$ by $\Sigma_\rho \times (X^s)^N$, $s>\frac 32$. In Section \ref{S:VP}, we will construct steady concentrated vortex patches near non-degenerate critical points of $H_{\vec{\CC}} (X)$ based on the above lemma. 

In order to study the spectral properties of the linearized vortex patch dynamics at steady states in Section \ref{S:stability}, we first write down the dynamic equation of the vortex patches parametrized by conformal mappings. Let 
\[
\big(X(t), \beta(t)\big) = (x_1, \ldots, x_N, \beta_1, \ldots, \beta_N) (t) \in \Sigma_\rho \times (X^s)^N, 
\] 
along with $\wt \Gamma_j(t)$, $\Gamma_j(t)$, and $\Omega_j(t)$ defined in \eqref{E:CM-2-1} and \eqref{E:CM-2-2}, correspond to a vortex patch solution to the Euler equation. Since $\Gamma_j$ defined on $B_1$ is conformal, it is clear 
\[
N_j \circ \Gamma_j =|\p_z \Gamma_j|^{-1} z \p_z \Gamma_j|_{S^1}, \quad T_j\circ \Gamma_j = |\p_z \Gamma_j|^{-1} i z \p_z \Gamma_j|_{S^1}. 
\]
Since $\p_t \Gamma_j$ yields the same normal velocity along $\p \Omega_j$ as in \eqref{E:VP-NV}, we have, for each $j$, 
\be \label{E:NV-1}
\p_t \Gamma_j \cdot (|\p_z \Gamma_j|^{-1} z \p_z \Gamma_j)  = (u\cdot N_j)\circ \Gamma_j = - ( \nabla_{T_j} \Psi)\circ \Gamma_j = |\p_z \Gamma_j|^{-1} \p_\theta \big(\Psi \circ \Gamma_j(e^{i\theta})\big),
\ee
along $S^1$. Due to the $\p_\theta$ in the above right side, the homogeneous space $\dot H^s (S^1)$ would get involved naturally. 
Differentiating the definition of $\Gamma_j$, we have 
\be \label{E:NV-2} \begin{split}
&\p_t \Gamma_j = \p_t x_j + r_j a_1^j \p_t \wt \Gamma_j + r_j \p_t a_1^j (z+ \wt \Gamma_j), \\
&\p_t a_1^j = - \sqrt{\pi}  \big(\int_{B_1} 1 + |\p_z \wt \Gamma_j|^2 d\mu\big)^{-\frac 32} \int_{B^1} \p_z \wt \Gamma_j \cdot \p_{tz} \wt \Gamma_j d\mu
\end{split} \ee
where $\p_t \wt \Gamma_j$ is defined by $\p_t \beta_j$ as in \eqref{E:CM-2-2}. 

We shall see that $\p_t X$ and  $\p_t \beta$ can be uniquely recovered from $\p_\theta \big(\Psi \circ \Gamma_j(e^{i\theta})\big)$. Clearly in this process domains $\Omega_j$ do not interact and we only consider single domains.

\begin{lemma} \label{L:Qbdd}
For $r>0$ and $\beta \in B_{X^s, R_s}$, $s>\frac 32$, where $B_{X^s, R_s}$ is the ball of radius $R_s$ in $X^s$ and $R_s$ defined in \eqref{E:Rs}, let holomorphic functions $\wt \Gamma$ and $\Gamma $ (including $a_1$) be defined by $x=0$, $\beta$, and $r$, as in \eqref{E:CM-2-1}, \eqref{E:CM-2-2}, and \eqref{E:a1j}. Then the linear mapping 
\be \label{E:B-1}
(y, \alpha) \in \BFR^2 \times X^{s'} \longrightarrow f= \dot \Gamma \cdot (  z \p_z \Gamma)|_{S^1}  \in \dot H^{s'} (S^1), 
\ee
is an isomorphism for any $s'\in [1-s, s-1]$, where 
\begin{subequations} \label{E:B-2}
\be \label{E:B-2a}
\dot \Gamma  = y  + r a_1 \dot{\wt \Gamma}  + r \dot a_1 (z+ \wt \Gamma )
\ee
\be \label{E:B-2b} \dot a_1 = - \sqrt{\pi}  \big(\int_{B_1} 1 + |\p_z \wt \Gamma |^2 d\mu\big)^{-\frac 32} \int_{B^1} \p_z \wt \Gamma  \cdot \p_{z} \dot {\wt \Gamma}  d\mu
\ee 
\end{subequations}
and the holomorphic function $\dot{\wt \Gamma}$ is determined by $\alpha $ as in \eqref{E:CM-2-2}. 
Moreover, the isomorphism $Q(r, \beta) \in L( \dot H^{s'}(S^1), \BFR^2 \times X^{s'})$ defined by $f \to Q(r, \beta)f= (y, \alpha)$ depends smoothly on $\beta \in B_{X^s, R_s}$ and 
\begin{align*}
&Q(r, \beta) = diag (r^{-1}, r^{-2}) Q(1, \beta), 
\\
& Q(1, 0) \sum_{l=1}^\infty (A_l \cos l \theta + B_l \sin l \theta) = \big((A_1, B_1), \sum_{l=3}^\infty (A_{l-1} \cos l \theta + B_{l-1} \sin l \theta)        \big). 
\end{align*}
\end{lemma}

\begin{proof} 
It is straight forward to obtain the scaling property in $r$ and thus we shall only work with $r=1$. We first show that $f$ defined in \eqref{E:B-1} satisfies $\int_{S^1} f d\theta=0$ if and only if $\dot a_1$ is given by \eqref{E:B-2b}. In fact, one can compute 
\[
\int_{S^1} f d\theta = \int_{S^1} \dot \Gamma \cdot (e^{i\theta} \p_z \Gamma) d\theta = \int_{S^1} \big( y + a_1 \dot {\wt \Gamma} +\dot a_1 (e^{i\theta} + \wt \Gamma) \big) \cdot  a_1 \big(e^{i\theta}(1  + \p_z \wt \Gamma) \big)d\theta. 
\]
Since along $S^1$, $\wt \Gamma$ and $\dot {\wt \Gamma}$ do not contain $e^{il\theta}$, $l=0, 1, 2$, in their Fourier series, we obtain 
\[
a_1^{-1} \int_{S^1} f d\theta  = 2 \pi \dot a_1 +  \int_{S^1} ( \dot a_1 \wt \Gamma + a_1 \dot {\wt \Gamma}) \cdot (e^{i\theta} \p_z \wt \Gamma) d\theta. 
\]
Given any holomorphic functions $\gamma_1(z)$ and $\gamma_2(z)$ on $B_1$, one may verify by using the Cauchy-Riemann equations and the divergence theorem 
\[
\int_{S^1} \gamma_1 \cdot (e^{i\theta} \p_z \gamma_2) d\theta = \int_{S^1} (\gamma_1 \overline{\p_z \gamma_2})\cdot e^{i\theta}  d\theta = \int_{B_1} \nabla \cdot  (\gamma_1 \overline{\p_z \gamma_2})  d\mu = 2 \int_{B_1} \p_z \gamma_1 \cdot \p_z \gamma_2   d\mu. 
\] 
which implies 
\[
\frac 1{2a_1} \int_{S^1} f d\theta  = \pi \dot a_1 + \int_{B_1} ( \dot a_1 \p_z \wt \Gamma + a_1  \p_z \dot {\wt \Gamma}) \cdot  \p_z \wt \Gamma d\mu .
\]
Solving for $\dot a_1$  and using \eqref{E:a1j} we obtain that $\int_{S^1} f d\theta=0$ if and only if $\dot a_1$ satisfies \eqref{E:B-2b}. 

Clearly $(y, \alpha) \to f$ is a bounded linear mapping from $\BFR^{2N} \times X^{s'}$ to $\dot H^{s'}(S^1)$. To complete the proof of the lemma, we show that the inverse $Q(1, \beta)f$ is well-defined. As $\p_z \Gamma  \ne 0$, it is easy to see that \eqref{E:B-1} is equivalent to 
\be \label{E:B-3}
\text{Re}\frac {\dot \Gamma }{z \p_z \Gamma } = \frac {f }{|\p_z \Gamma |^2}, \; \text{ on } S^1. 
\ee
Expand the the above right side in a Fourier series
\[
\frac {f (\theta)} {|\p_z \Gamma  (e^{i\theta})|^{2}} = \wt A_0 + \wt A_1 \cos \theta + \wt B_1 \sin \theta + \sum_{l=2}^\infty \big(\wt A_l \cos (l \theta) + \wt B_l \sin (l \theta) \big) \triangleq \wt A_0 + \wt A_1 \cos \theta + \wt B_1 \sin \theta + \wt f  (\theta). 
\]
Much as in \eqref{E:CM-2-2}, let 
\[
F(z) = \sum_{l=2}^\infty (\wt A_l - i \wt B_l) z^l, \; \; \text{ then } \;  \text{Re}F  (e^{i\theta}) = \wt f  (\theta). 
\]
Let $\wt F = \dot \Gamma  (z \p_z \Gamma )^{-1} - F $ and \eqref{E:B-3} is satisfied if and only if 
\[
\text{Re} \wt F  (e^{i\theta} )= \wt A_0 +\wt A_1 \cos \theta + \wt B_1 \sin \theta.
\]
Since $\wt F $ is holomorphic on $B_1 \backslash \{0\}$ with a possible pole of order $O(z^{-1})$ at $0$,  $\wt F $ should take the form of $q_{-1} z^{-1} + q_0 + q_1 z$. As 
\[
0< \p_z \Gamma (0) = a_1 = 1  + O(|\beta|_{H^s}^2),  \quad \p_{zz} \Gamma  (0)=0, \quad \p_z \dot \Gamma (0) \in \BFR, \quad \p_{zz} \dot \Gamma (0) =0,  
\]
a simple calculation of power series product implies 
\[
q_{-1} =\frac {\dot \Gamma (0)}{\p_z \Gamma  (0)}, \quad  q_1 = - \frac {\p_{zzz} \Gamma  (0) \dot \Gamma (0)}{2 \big(\p_z \Gamma  (0)\big)^2}, \quad \wt A_0=q_0 = \frac {\p_z \dot \Gamma (0)} {\p_z \Gamma  (0)}   \in \BFR.  
\] 
Through straight forward linear algebra calculation, we obtain that $\dot \Gamma(0)$ can be uniquely determined by $\wt A_1$ and $\wt B_1$ if $|\p_{zzz} \Gamma  (0)| < 2 |\p_{z} \Gamma  (0)|$, which is ensured by the assumption $|\beta|_{H^s} < R_s$. Therefore 
one may solve $\wt F $, $\dot \Gamma  (0) \in \BFC$, and $\p_z \dot \Gamma (0) \in \BFR$ from $\wt A_{0, 1}$ and $\wt B_0$ (and thus from $f $)
and thus we have found the unique solution $\dot \Gamma  = z \p_z \Gamma  (\wt F  + F )$ to equation \eqref{E:B-1}, which in turn yields $y $ and $\alpha $. 
The smooth dependence of $y$ and $\alpha$ on $x$ and $\beta$ is clear from the above process. When $\beta=0$, we have $\Gamma (z) = z$ and the exact form of $Q(1, 0) f$ can be computed explicitly using the form of $F$ and $q_{-1} = \wt A_1 + i \wt B_1$ and $q_0=q_1=0$. 
\end{proof}

Geometrically, $(y, \alpha) = Q(r, \beta) f$ corresponds to a variation $\dot \Gamma$ of the domain $\Gamma(B_1)$ recovered from the weighted normal velocity $f$. 
Equality \eqref{E:B-2b} on $\dot a_1$ means the perturbed domain $(\Gamma + \tau \dot \Gamma) (B_1)$, which has normal velocity $|\p_z \Gamma|^{-1} f$  at $\tau=0$ is volume preserving infinitesimally at $\tau=0$. 

Consequently the vortex patches dynamics of the Euler equation in the conformal mapping parametrization takes the form 
\be \label{E:Euler-VP-CM}
\p_t (x_j, \beta_j) = Q(r_j, \beta_j) \p_\theta \big(\Psi \circ \Gamma_j(e^{i\theta})\big)
\ee 
where $\Psi$ is the stream function of $\omega$ determined by Lemma \ref{L:Stream-F}. In Section \ref{S:stability}, we shall linearize this equation at a steady state to study its spectral stability. In this analysis, the following Hamiltonian structure of \eqref{E:Euler-VP-CM} essentially inherited from that of \eqref{E:Euler-V} plays an important role. 

In fact, at any $(X, \beta) \in \Sigma_\rho \times (X^s)^N$ with $\tilde \Gamma_j$, $\Gamma_j$, and $\Omega_j$ defined accordingly,  using \eqref{E:energy-V} the conserved energy $\BFE$ can be computed in terms of $\Omega_j$:
\be \label{E:energy-p} \begin{split} 
E_p (X, \beta) -  \frac 12  \vec{\CC}^T \CN^{-1}& \vec{\CC} =  - \frac 12 \int_\Omega \omega \Delta_0^{-1} \omega d\mu -  \sum_{j=1}^N \sum_{l=1}^n \frac {c_l \mu_j}{\pi r_j^2} \int_{\Omega_j} \CH_l d \mu \\
=& - \sum_{j, k=1}^N \frac {\mu_j \mu_k}{2\pi^2 r_j^2 r_k^2} \int_{\Omega_j} \Delta_0^{-1} \chi_{\Omega_k} d\mu -  \sum_{j=1}^N \sum_{l=1}^n \frac {c_l \mu_j}{\pi r_j^2} \int_{\Omega_j} \CH_l d \mu. 
\end{split} \ee

At each $\omega$ in the form of \eqref{E:VP-omega} associated to $(X, \beta)$ (thus equivalently $\wt \Gamma_j$, $\Gamma_j$, $\Omega_j$), a variation $(\dot X, \dot \beta)$ of $(X, \beta)$ corresponds to a variation (in the distribution sense) $\dot \omega$  of vortex patches consists of a measure supported on each $\p \Omega_j$ with the density proportional to the normal variation $\nu_j: \p \Omega_j \to \BFR$ of each $\p \Omega_j$
\be \label{E:tangent-p}
\dot \omega = \sum_{j=1}^N \frac {\mu_j}{\pi r_j^2} \nu_j\delta_{\p \Omega_j}, \qquad \int_{\p \Omega_j} \nu_j dS =0, 
\ee
where 
the condition $\int_{\p \Omega_j} \nu_j dS =0$ is due to the volume preservation of the fluid transport.  
As in the proof of Lemma \ref{L:Ham-E}, the derivative of $E_p$ can be calculated as 
\[
\langle \CD E_p(X, \beta), (\dot X, \dot \beta)\rangle 
=  - \int_\Omega \dot \omega \Psi d\mu =  - \sum_{j=1}^N \frac {\mu_j}{\pi r_j^2} \int_{\p \Omega_j} \nu_j (\Psi -const) d S.
\]
Due to Lemma \ref{L:Qbdd}, we have 
\[
(\dot x_j, \dot \beta_j) = Q(r_j, \beta_j) (|\p_z \Gamma_j| \nu_j) 
\]
and thus 
\be \label{E:d-energy-p}  \begin{split}
\langle \CD E_p(X, \beta), (\dot X, \dot \beta)\rangle = - \sum_{j=1}^N \frac {\mu_j}{\pi r_j^2} \int_{S^1} Q(r_j, &\beta_j)^{-1} (\dot x_j, \dot \beta_j)\times \\
&\big( \Psi\circ \Gamma_j  - \frac 1{2\pi} \int_{S^1}  \Psi \circ \Gamma_j  d\theta' \big)   d\theta. 
\end{split} \ee
Through the duality of $L^2 (S^1)^N$, one may identify $\CD E_p (X,\beta)$ with 
\be \label{E:d-energy-p-1}  
\CD E_p(X, \beta) = (Y, \alpha), \quad (y_j, \alpha_j) =  - \frac {\mu_j}{\pi r_j^2} \big(Q(r_j, \beta_j)^{-1} \big)^*  
\big( \Psi\circ \Gamma_j  - \frac 1{2\pi} \int_{S^1}  \Psi \circ \Gamma_j  d\theta' \big).
\ee
Since the vortex dynamics of the Euler equation is equivalent to \eqref{E:Euler-VP-CM}, from Lemma \ref{L:Qbdd}, we obtain its form in the conformal mapping parametrization
\be \label{E:VP-Ham-1} 
\p_t (X, \beta) = J_p(\beta) \CD E_p (X, \beta), 
\ee
where 
\[
J_p (\beta) = diag\big(J_p^1(\beta_1), \ldots, J_p^N (\beta_N)\big) , \quad J_p^j (y_j, \alpha_j) =- \mu_j^{-1} \pi r_j^2 Q(r_j, \beta_j) \p_\theta \big(Q(r_j, \beta_j)^* (y_j, \alpha_j)\big).
\]
Clearly $J_p$ is (unbounded, depending on $\vr, \beta$) anti-self-adjoint operator, namely
\be \label{E:J-VP} 
J_p^* = -J_p. 
\ee
This symplectic operator $J_p(\beta)$ for vortex patch dynamics is much nicer than $J_0(\omega)$ for the general 2-dim Euler equation as $J_p$ has a bounded inverse on $\dot H^s(S^1)$. 

 

This nonlinear Hamiltonian structure induces a linear one for the  vortex patch dynamics  linearized at a steady states, which will be useful in the  the spectral analysis  in Section \ref{S:stability}.


\section{Concentrated steady vortex patches} \label{S:VP}

In this section, we shall prove Theorem \ref{T:VP} on the existence of steady piecewise constant concentrated  vorticities $\omega = \Sigma_{j=1}^N \omega_j$ in the form of \eqref{E:VP-omega} satisfying \eqref{E:localized-V} and \eqref{E:localization}. 
Fix vorticities $\mu_1, \ldots, \mu_N \in \BFR\backslash\{0\}$ as given in Theorem \ref{T:VP},  $s>\frac 32$, and 
\be \label{E:rho-1} 
\rho>0, \quad 
\; \text{ such that } X_* \in \Sigma_\rho,
\ee
where we recall $X_*$ is the critical point of $H_{\vec{\CC}}(X)$ provided in Theorem \ref{T:VP}. For any 
\[
X =(x_1, \ldots, x_N) \in \Sigma_\rho, \quad \beta = (\beta_1, \ldots, \beta_N) \in B_{ (X^s)^N, R_s}, 
\]
and 
\be \label{E:D-R}
\vec{r} \in D_R, \;\text{ where } D_R = \{\vec{r}= (r_1, \ldots, r_N)^T \in (\BFR^+)^N \mid |\vec{r}| < R\}, \quad R<< \rho, 
\ee
where $X^s$ and $R_s$ are defined in \eqref{E:Xs} and \eqref{E:Rs} and $B_{(X^s)^N, R_s }$ is the ball of radius $R_s$ in $(X^s)^N$, define $\wt \Gamma_j$, $\Gamma_j$, and $\Omega_j = \Gamma_j(B_1)$ as in \eqref{E:CM-2-1}, \eqref{E:CM-2-2}, and \eqref{E:a1j}, as well as the stream function $\Psi$ by Lemma \ref{L:Stream-F}. According to Lemma \ref{L:steady-VP}, $(X, \beta)$ corresponds to a steady vortex patch if and only if $\Psi|_{\p \Omega_j} =const$ for $j=1, \ldots, N$. Define 
\be \label{E:phi-VP-1}
\phi_j (X, \beta, \vec{r}) (\theta) =\p_\theta  \Psi \big(\Gamma_j(\theta)\big). 
\ee
For small $\vec{r}$, to make $\phi_j= 0$ for all $1\le j \le N$, we shall seek proper $X \in \Sigma_\rho$ and $\beta \in (X^s)^N$ through a perturbation argument carried in the rest of the section. 

As seen in the definition of $\Psi$ in Lemma \ref{L:Stream-F}, a key part is the convolution with $\log |x|$ with each individual vortex patch. For $r>0$ and $\alpha \in X^s$ with $|\alpha|< R_s$, let 
\[
\wt h(\alpha)(x) = \frac 1{2\pi^2 r^2}  \int_{D} \log |x - y|  d\mu_y, \;\; x\in \BFR^2; \quad h(\alpha) (\theta)=\p_\theta \big( \wt h (\alpha) \big( \gamma (e^{i\theta})\big)\big) 
\]
where $D =\gamma(B_1)$ and $\gamma(z) = a_1r \big(z + \wt \gamma(z) \big)$ is the conformal mapping centered at $0$ defined by $\alpha$ and $r$ as in \eqref{E:CM-2-1}, \eqref{E:CM-2-2}, and \eqref{E:a1j}. The coordinate change $x =\gamma(z)$ 
yields 
\be \label{E:h-1} \begin{split}
h(\alpha) (\theta) = &\frac {1}{2\pi \int_{B_1} 1+ |\p_z \wt \gamma|^{2}d\mu } \int_{B_1} |1+ \p_z \wt \gamma(z)|^{2} \frac { \big(\gamma (e^{i\theta}) - \gamma(z)\big) \cdot  \p_\theta \big(\gamma (e^{i\theta}) \big)} {|\gamma (e^{i\theta}) - \gamma(z)|^2} d\mu_z \\
=& \frac {1}{2\pi \int_{B_1} 1+ |\p_z \wt \gamma|^{2}d\mu } \text{Re}\int_{B_1} |1+ \p_z \wt \gamma(z)|^{2} \frac {  ie^{i\theta}  \big(1+ \p_z \wt \gamma (e^{i\theta}) \big)}{ e^{i\theta} - z +\wt\gamma  (e^{i\theta}) -\wt \gamma(z)} d \mu_z
\end{split} \ee
which is independent of $r$ due to the scaling property. We first have 

\begin{lemma} \label{L:smooth-2} 
It holds $h \in  C^\infty \big(B_{X^s, R_s}, \dot H^{s-1} (S^1) \big)$ for any $s>\frac 32$, $h(0)=0$, and 
\[
\CD h(0) \sum_{l=3}^\infty (A_l \cos l \theta + B_l \sin l \theta) =\sum_{l=2}^\infty  \frac {l-1}{2\pi} \big(B_{l+1} \cos l \theta -A_{l+1} \sin l \theta\big).
\]
Moreover, for any $s' \in [1-s, s]$, $\CD h \in C^\infty \big(B_{X^s, R_s},  L(X^{s'}, \dot H^{s'-1} (S^1)) \big)$. 
\end{lemma}

Clearly $\wt h(\alpha)$ is the stream function of $\omega = \frac 1{|D|}\chi_D$ on $\BFR^2$. The formal calculation on $\CD h(0)$ had first carried out by Kelvin (1880) (see, for example, \cite{La93}) who also obtained the linear stability of circular patches in $\BFR^2$. 

\begin{remark} \label{R:smooth}
It is worth pointing out that $\frac 1{|D|}\chi_D$ is not in $C^1(\Omega)$ due to its jump along $\p D$, so $\wt h(\alpha)$ is not smoother than $C^3(\BFR^2)$. This lemma states that $\wt h(\alpha)|_{\p D}$ can be as smooth as $\p D$, essentially due to the observation that the obstacle to the smoothness of $\wt h(\alpha)$ is mainly just the differentiation in the transversal direction of $\p D$.
\end{remark}

\begin{proof} 
The regularity of $h(\alpha)$ and its smooth dependence on $\alpha$ may be proved by analyzing the above singular integral operator. Alternatively, the following is a slight modification of the proof of Lemma B.3 in \cite{SWZ13}. 

For any domain $U \subset \BFR^2$, let $(\Delta_U)^{-1}$, $\BBH_U$, $N_U$, and $\BBN_U$ denote the inverse Laplacian on $U$ with zero Dirichlet boundary condition, the harmonic extension from $\p U$ to $U$, the unit outward normal vector of $U$, and the Dirichlet-to-Neumann operator associated to harmonic functions on $U$. Namely, $\BBN_U f = \nabla_{N_U} (\BBH_U f) :\p U\to \BFR$. If $U$ is unbounded, $\BBH_U f \in \dot H^1(U)$ is required.  
It is clear 
\[
\wt h(\alpha) \in W_{loc}^{2, p} (\BFR^2), \;\; \forall p \in (1, \infty), \quad \;\; \Delta \wt h(\alpha)|_{D} = 1, \quad \Delta \wt h(\alpha)|_{D^c} =0. 
\]
Let $f (\alpha) = \wt h(\alpha)|_{\p D} \in H^{\frac 32} (\p D)$, then we have 
\[
\wt h(\alpha)|_{D^c} = \BBH_{D^c} f (\alpha), \quad \wt h(\alpha)|_D = \BBH_D f(\alpha) + \Delta_D^{-1} 1. 
\]
Since $\wt h(\alpha) \in W_{loc}^{2,p}(\BFR^2)$ implies 
\[
0=\nabla_{N_D} \wt h(\alpha)|_{\p D}+ \nabla_{N_{D^c}} \wt h(\alpha)|_{\p D} = \BBN_D f(\alpha) + \nabla_{N_D} \Delta_D^{-1}1 + \BBN_{D^c} f(\alpha),
\]
and $\BBN_D + \BBN_{D^c}$ is uniformly positive, we have 
\[
f(\alpha) = - (\BBN_{D} + \BBN_{D^c})^{-1} \nabla_{N_D}   \Delta_D^{-1} 1 \in H^s(\p D). 
\]
Moreover, the $C^\infty$ dependence of $f(\alpha)\circ \gamma |_{S^1}$ on $\alpha$, as well as that of $h(\alpha) (\theta) = \p_\theta \big(f(\alpha) \big( \gamma (e^{i\theta})\big)\big)$, 
follow from the properties of $\BBN$ and $\Delta^{-1}$, in particular their smooth dependence on domains. (See, for example, \cite{SZ11} where are the variations of these operators with respect to domains are given.) 

When $\alpha=0$, $D$ is a disk and $\wt h(0)$ is radially symmetric and thus $h(0) \equiv 0$. Having proved the smoothness of $h(\alpha)$ with respect to $\alpha$, one may compute $\CD h(0) \alpha$ based on \eqref{E:h-1}:
\begin{align*}
&\big( \CD h(0) \alpha\big) (\theta) =  \frac 1{2\pi^2}\text{Re}\int_{B_1} ie^{i\theta} \big( \frac {2 \text{Re}(\p_z \wt \gamma(z))} {e^{i\theta} - z}  +  \frac {\p_z \wt \gamma (e^{i\theta})}{e^{i\theta} - z}  - \frac { \wt\gamma  (e^{i\theta}) -\wt \gamma(z)}{ (e^{i\theta} - z)^2}\big) d \mu_z\\
=&\frac 1{2\pi^2} \text{Re}\int_{B_1} i \big(\frac {2 \text{Re}(\p_z \wt \gamma(e^{i\theta}z))} { 1- z} + \frac {\p_z \wt \gamma (e^{i\theta})}{1 - z}  - e^{-i \theta} \frac { \wt\gamma  (e^{i\theta}) -\wt \gamma(e^{i\theta} z)}{ (1 - z)^2}\big) d \mu_z \triangleq I + II + III.
\end{align*}
For $\alpha = A \cos l \theta + B \sin l \theta$, $l\ge 3$, according to \eqref{E:CM-2-2} we have $\wt \gamma (z) = r_* e^{i\theta_*} z^{l}$, where $r_* e^{i\theta_*}= A- iB$. One may compute 
\begin{align*}
I=& \frac {l r_*}{\pi^2} \text{Re}\int_{B_1} i \sum_{m=0}^\infty z^m \text{Re}(e^{i( \theta_* + (l-1) \theta)} z^{l-1}) d\mu_z \\
=& \frac {l r_*}{\pi^2} \text{Re}\int_{0}^1 \int_{S^1} i \tau^{2l-1} e^{i(l-1) \theta'} \cos \big( \theta_* + (l-1) \theta+ (l-1) \theta' \big) d\theta' d\tau  \\
=& - \frac {l r_*}{\pi^2} \int_{0}^1 \int_{S^1} \tau^{2l-1} \sin (l-1) \theta' \cos \big( \theta_* + (l-1) \theta+ (l-1) \theta' \big) d\theta' d\tau\\
=& \frac {r_*}{2\pi}  \sin \big( \theta_* + (l-1) \theta \big) = \frac 1{2\pi} \big( - B \cos (l-1) \theta + A \sin (l-1) \theta\big) 
\end{align*}
and 
\[
II+ III= \frac 1{2\pi^2}\text{Re}\int_{B_1} \frac {i (A - iB)e^{i(l-1)\theta}}{1-z} \big(l   - \sum_{m=0}^{l-1} z^m     \big) d \mu_z=  \frac {l-1}{2\pi} \big(B\cos (l-1) \theta - A\sin (l-1) \theta\big).
\]
The desired formula of $\CD h(0) \alpha \in \dot H^{s-1}(S^1)$ follows. 
\end{proof}

To study other terms in $\phi_j (X, \beta, \vec{r}) (\theta)$ defined in \eqref{E:phi-VP-1}, for any given smooth function $f(x, y)$ defined on $\Omega^2$ and $1\le j, j' \le N$, let 
\be \label{E:F-1}
F(X, \beta, \vec{r}) (\theta) = r_j^{-1} \p_\theta \int_{\Omega_{j'}}f \big( \Gamma_j (e^{i\theta}), y\big) \omega_{j'} (y) d\mu_y, \quad \omega_{j'} = \frac {\mu_{j'}}{\pi r_{j'}^2} \chi_{\Omega_{j'}}. 
\ee
One may compute 
\be \label{E:F} \begin{split}
F(X, \beta, \vec{r}) (\theta)= \frac {\mu_{j'} (a_1^{j'})^2 a_1^{j}}{\pi}  \int_{B_1} &|1+ \p_z \wt \Gamma_{j'}|^{2}  \big((\nabla_1f) \big(\Gamma_j (e^{i\theta}), \Gamma_{j'} (z)\big)\big)\\
& \cdot \big( ie^{i\theta} (1+ \p_z \wt \Gamma_j (e^{i\theta}) )\big) d\mu_z 
\end{split} \ee
which immediately implies the regularity of $F$ and its smooth dependence on $(X, \beta, \vec{r})$. 

\begin{lemma} \label{L:F-1}
For any 
$s' \in [1-s, s]$ and smooth $f(x,y)$, there exists $R>0$ such that 
$F$ can be extended to $\vec{r} \in \overline{ D_R}$ satisfying   
\begin{align*}
& \CD_{(X, \beta)} F \in C^\infty \big(\Sigma_\rho  \times B_{(X^s)^N, R_s} \times \overline{D_R}, \ L \big( \BFR^{2N} \times (X^{s'})^N, \dot H^{s'-1}(S^1)\big)\big).
\end{align*}
Moreover we have 
\begin{align*}
F(X, \beta, 0)=  \lim_{|\vec{r}| \to +0} F (X, \beta, \vec{r}) =  \mu_{j'} a_1^j \big(i e^{i\theta} (1+ \p_z \wt \Gamma_j (e^{i\theta}) )\big)\cdot  \nabla_1 f(x_j, x_{j'}).
\end{align*}
\end{lemma}

In the above last formula of $F(X, \beta, 0)$, where both $a_1^j$ and $\wt \Gamma_j$ depend on $\beta_j$, follows from the fact that $\Gamma_m (z) =x_m$ if $r_m=0$. Like $\CD h$, the above property of  $\CD F$ means that not only $F$ is smooth from $\beta \in (X^s)^N$ to $\dot H^{s-1}(S^1)$, but also $\CD F$ may be applied to functions of lower regularity. This will be used in Section \ref{S:stability} in the spectral analysis of the linearized vortex patch dynamics at the steady patches. Based on this lemma, we can prove 

\begin{lemma} \label{L:smooth-3}
For any 
$s' \in [1-s, s]$, there exists $R>0$ such that 
\[
\phi_j(X, \beta, \vec{r})= r_j R_j (X, \beta, \vec{r}) + \mu_j h(\beta_j) 
\]
and 
\begin{align*}
&R_j \in C^\infty \big(\Sigma_\rho  \times B_{(X^s)^N, R_s} \times \overline{D_R}, \dot H^{s-1}(S^1)\big),\\
&\CD_{(X, \beta)} R_j  \in C^\infty \big(\Sigma_\rho  \times B_{(X^s)^N, R_s} \times \overline{D_R}, \ L \big( \BFR^{2N} \times (X^{s'})^N, \dot H^{s'-1}(S^1)\big)\big). 
\end{align*}
Moreover we have 
\be \label{E:Rj}
R_j(X, \beta, 0)=  \lim_{|\vec{r}| \to +0} R_j (X, \beta, \vec{r}) = 
- \mu_j^{-1} a_1^j \big(i e^{i\theta}(1+ \p_z \wt \Gamma_j (e^{i\theta}) )\big)\cdot  \nabla_{x_j} H_{\vec{\CC}} (X)
\ee 
where $H_{\vec{\CC}}$ was defined in \eqref{E:E0}. 
\end{lemma}

\begin{proof}
Firstly, due to the harmonicity and thus the smoothness of $\CH_l(x)$ in $\Omega$, clearly 
\be \label{E:smooth-temp-1}
r_j^{-1} \p_\theta \CH_l \big( \Gamma_j( e^{i\theta}) \big)  = a_1^j \big((\nabla \CH_l) \big(\Gamma_j(e^{i\theta})\big)\big) \cdot \big( ie^{i\theta} (1+ \p_z \wt \Gamma_j (e^{i\theta}) )\big)   \in \dot H^{s-1} (S^1)
\ee
and thus it depends on $\vec{r} \in \overline{D_R}$, $X$, and $\beta$ smoothly. Moreover the variation of above formula with respect to $\beta$ is clearly a bounded linear operator from $(X^{s'})^N$ to $\dot H^{s'-1} (S^1)$ with coefficients depending on $(X, \beta, \vec{r}) \in \Sigma_\rho  \times B_{(X^s)^N, R_s} \times \overline{D_R}$ smoothly. 

According to Lemma \ref{L:Stream-F}, the other terms in the definition of $\phi_j$ are all in the form of \eqref{E:F-1} with $f(x, y)$ given by $\BFG_0(x, y)$ for $x \in \overline{\Omega_j}\ne \overline{\Omega_{j'}} \ni y$, $\CH_{l_1} (x) \CH_{l_2}(y)$, and $\tilde g(x, y)$. Therefore Lemmas \ref{L:smooth-2} and \ref{L:F-1} imply the regularity of $R_j$ and its smooth dependence on $(X, \beta, \vec{r}) \in \Sigma_\rho  \times B_{(X^s)^N, R_s} \times \overline{D_R}$. 

Finally $\Gamma_m (z) =x_m$ for all $1\le m\le N$ if $\vec{r}=0$. Lemma \ref{L:F-1} and \eqref{E:smooth-temp-1} imply 
\begin{align*}
R_j (X,  \beta, 0) = &a_1^j \big(i e^{i\theta}(1+ \p_z \wt \Gamma_j (e^{i\theta}) )\big)\cdot  \Big(\sum_{l=1}^n c_l \nabla \CH_{l} (x_j)  - \sum_{ j'\ne j, 1\le j'\le N} \mu_{j'}  \nabla_1 \BFG_0 (x_j, x_{j'}) \\
&- \mu_{j} \nabla_1 \tilde g(x_j, x_j)  
- \sum_{l_1, l_2=1}^n  \CN^{l_1l_2} \mu_j \nabla \CH_{l_1} (x_j) \CH_{l_2} (x_j) \Big). 
\end{align*} 
The symmetry of $\tilde g$, $\BFG_0$, and $\CN^{l_1, l_2}$ implies that the big parentheses on the above right side is equal to $-  \nabla_{x_j} H_{\vec{\CC}} (X)$ and this completes the proof of the lemma. 
\end{proof}

With the above preparations, we are ready to prove Theorem \ref{T:VP} on steady concentrated vortex patches. Let
\be \label{E:CF-VP-1} 
\CF (X, \beta, \vec{r}) = (\CF_1, \ldots, \CF_N), \quad \;  
\ee
where 
\[
\CF_j(X, \beta, \vr) =  r_j^{-1} \phi_j (X, M_{\vr} \beta, \vec{r}), 
\quad(M_{\vr})_{N\times N}= diag (r_1, \ldots, r_N).
\]
The different scalings in $X$ and $\beta$ are due to the degeneracy of $\CD_X \phi_j|_{\vec{r}=0}$. 
In the rest of the section (and also in the next section), $\wt \Gamma_j$ in the definition of $\phi_j(X, M_{\vr} \beta, \vec{r})$ is defined as in \eqref{E:CM-2-2}, but by $r_j \beta_j$. Namely Re$\wt \Gamma_j|_{S^1} = r_j \beta_j$. The conformal mapping $\Gamma_j$ and $\Omega_j = \Gamma_j(B_1)$ are defined accordingly. Due to the different scaling in $\CF$, we split the target space $\big(\dot H^s(S^1)\big)^N$ into the direct sum $ Y_1 \oplus (Y^s)^N $ of two subspaces orthogonal with respect to $L^2 (S^1)$ where 
\be \label {E:Ys-1} 
Y_1 =  (\BFR^N \{\cos \theta\}) \oplus (\BFR^N\{\sin\theta\}) \sim \BFR^{2N}, \quad Y^s =  \{\cos\theta, \sin \theta\}^{\perp_{L^2(S^1)}} \subset \dot H^s (S^1)
\ee
and accordingly split vector-valued functions in $\big(\dot H^s(S^1)\big)^N$  
\[
(f_1, \ldots, f_N) =\big( (\cos \theta, \sin \theta) y_1, \ldots, (\cos \theta, \sin \theta) y_N \big) +  (\wt f_1, \ldots, \wt f_N), \quad y = (y_1^T, \ldots, y_N^T)^T \in \BFR^{2N}. 
\]
Lemma \ref{L:smooth-3} implies 

\begin{lemma} \label{L:DCF-2}
$\CF \in C^\infty \big(\Sigma_\rho  \times B_{(X^s)^N, R_s} \times \overline{D_R}, \dot H^{s-1}(S^1)\big)$. Moreover, 
$\CD_\beta \CF (X, \beta, 0)$ is isomorphic from $(X^s)^N$ to $(Y^{s-1})^N$ and $\CD_X \CF(X, \beta, 0)(\BFR^{2N})\subset Y_1$ and  
\be \label{E:DCF-1} \begin{split}
&\CD_\beta \CF (X, \beta, 0) \alpha = (\mu_1 \CD h(0) \alpha_1, \ldots,  \mu_N\CD h(0) \alpha_N),  \\
& \CD_X \CF(X,\beta, 0) \wt X = \big( (\cos \theta, \sin \theta) y_1, \ldots, (\cos \theta, \sin \theta) y_N \big),
\end{split}\ee
where 
\[
y = (y_1^T, \ldots, y_N^T)^T = \Lambda^{-1} J_N D^2 H_{\vec{\CC}} (X) \wt X
\]
and $\Lambda$ and $J_N$ are defined in \eqref{E:symplectic-1}. 
\end{lemma}

Therefore we have 

\begin{corollary} \label{C:CF-zero-1}
It holds that $\CF(X_*, 0, 0)=0$ and $\big(\CD_{(X, \beta)} \CF(X_*, 0, 0)\big)^{-1} \in L\big( \big(\dot H^{s-1}(S^1)\big)^N, \BFR^{2N} \times (X^s)^N \big)$ if and only if $X_*= (x_1^*, \ldots, x_N^*)$ is a non-degenerate critical point of $H_{\vec{\CC}}(X)$.  
\end{corollary}

\noindent {\it Proof of Theorem \ref{T:VP}.} Following from the Implicit Function Theorem, for $|\vec{r}|<<1$, there exist $\big(X(\vec{r}), \beta(\vec{r})\big) \in \Sigma_\rho \times B_{(X^s)^N, R_s}$, which are $C^\infty$ in $\vec{r}$, such that 
\[
\CF \big(X(\vec{r}), \beta(\vec{r}), \vec{r} \big) =0, \qquad X(0) = X_*, \; \; \beta(0)=0.
\]
The smoothness of $\beta(\vec{r})$ yields $\beta(\vec{r})=O(|\vec{r}|)$. The corresponding vortical domains $\Omega_j$, $1\le j\le N$, are determined by \eqref{E:CM-2-1} where $\wt \Gamma_j$ satisfies Re$\wt \Gamma_j |_{S^1} = r_j \beta_j (\vec{r}) = O(|\vec{r}|r_j)$. Therefore $r_j^{-1} \p \Omega_j$ is an $O(r_j |\vec{r}|)$ perturbation to $S^1$ in $H^s$ topology. This argument can be carried out for any $s>\frac 32$. Due to the local uniqueness of the solutions, sufficiently small solutions belong to $\Sigma_\rho \times (X^s)^N$  for any $s>\frac 32$ and thus are independent of $s>\frac 32$ for small $\vr$. This completes the proof of Theorem \ref{T:VP}. 
\hfill $\square$

\vspace{.05in} The properties that $r_j^{-1} \p \Omega_j$ is an $O(r_j |\vec{r}|)$ perturbation to $S^1$ will be important in the spectral analysis in Section \ref{S:stability}. For $\vr = \ep\vr_0$ for some fixed $\vr_0$, the same property is also claimed in Corollary 1 in \cite{Wan88}, however, it is not very clear to us how this is obtained in the argument there. The direct application of the Implicit Function Theorem usually yields one order less in $\ep$. 

\begin{remark} \label{R:shape}
Using \eqref{E:F}, one may compute 
\[
\big(\CD_{\vr} \CF_j (X_*, 0, 0) \vr \big) (\theta) = \big(\CD_{\vr} R_j (X_*, 0, 0)\vr\big) (\theta) = \langle \vec{A}, \vr \rangle \cos 2\theta + \langle \vec{B}, \vr \rangle \sin 2\theta
\]
for some $\vec{A}, \vec{B}\in \BFR^N$. From the Implicit Function Theorem and Lemmas \ref{L:DCF-2} and \ref{L:smooth-2}, 
\begin{align*}
\big(\CD_{\vr} \beta_j (0) \vr \big) (\theta)= & -\Big( \big(\CD_\beta \CF_j (X_*, 0, 0)\big)^{-1} \CD_{\vr} \CF_j (X_*, 0, 0) \vr \Big) (\theta) \\
=& 2\pi \mu_j^{-1} ( \langle \vec{B}, \vr\rangle \cos 3\theta - \langle \vec{A}, \vr\rangle \sin 3\theta).
\end{align*}
Therefore we have from \eqref{E:CM-2-2}
\[
\Gamma_j (\vr)(z)= x_j (\vr) + a_1^j r_j ( z + r_j \langle \vec{Q}, \vr\rangle z^3\rangle) + O(r_j^2 |\vr|^2), \quad a_1^j = 1+ O(|r_j \vr|),
\] 
for some $\vec{Q} \in \BFC^N$. This means that $(a_1^jr_j)^{-1} \Omega_j$ is roughly the image of $B_1$ under a mapping $z+ O(|r_j\vr|) z^3$, which is basically like an ellipses. To see this, consider the image of $S^1$ under $z+\ep^2 z^3$ for small $\ep$. One only needs to notice that the boundary is parametrized by $R = 1 + 2\ep^2 \cos 2\theta +O(\ep^4)$, which satisfies $(\frac x{1+2\ep^2})^2 + (\frac y{1-2\ep^2})^2 =1 +O(\ep^4)$. 
\end{remark}

\section{Spectral stability of concentrated steady vortex patches} \label{S:stability} 

In this section, we shall study the linearized vortex patch dynamics at the steady states $\big(X_*(\vec{r}), M_{\vr} \beta_* (\vec{r})\big)$, $\vr \in \overline{D_R}$, found in Theorem \ref{T:VP}, where 
\[
X_*=(x_{*1}, \ldots, x_{*N}) (\vec{r}), \quad \beta_* (\vec{r}) = (\beta_{*1}, \ldots, \beta_{*N}) (\vec{r}) \in (X^s)^N, \;  \forall s, \quad \CF\big( X_*(\vec{r}),  \beta_* (\vec{r}), \vec{r}\big)=0, 
\]
and we shall follow the same notations as in Section \ref{S:VP} and Subsection \ref{SS:Ham-VP}. In this section, let $\wt \Gamma_{*j}(\vec{r})$ be defined as in \eqref{E:CM-2-2}, but by $r_j \beta_{*j} (\vec{r})$. Namely Re$\wt \Gamma_{*j} (\vec{r})|_{S^1} = r_j \beta_{*j} (\vec{r})$. The vorticity $\omega_*$, stream function $\Psi_* (\vr)$, the conformal mappings $\Gamma_{*j} (\vec{r})$, and $\Omega_{*j} (\vec{r}) = \Gamma_{*j} (\vec{r}) (B_1)$ are defined accordingly. For notational simplicity, we sometimes skip their $\vr$ dependence in this section. 

Let $A(\vr)$ denote the linearized operator of the vortex patch dynamics  at $\big(X_*(\vr), \beta_*(\vr) \big)$. Namely the linearized equation of \eqref{E:Euler-VP-CM} takes the form 
\be \label{E:LVP}
\p_t (X, \beta) = A(\vr) (X, \beta). 
\ee
According to \eqref{E:Euler-VP-CM}, $A(\vr)$ takes the form 
\be \label{E:Avr} \begin{split}
&A(\vr) = ( A_1, \ldots, A_N) (\vr), \quad A_j (\vr) \in L\big((\BFR^2 \times X^s)^N,  \BFR^2 \times X^{s-1}\big)\\
&A_j(\vr) (X, \beta) = Q\big(r_j, r_j \beta_{*j}(\vr)\big) \CD_{(X, \beta)} \phi_j \big( X_*(\vr), M_{\vr} \beta_*(\vr), \vr \big) (X, \beta) 
\end{split} \ee
where the boundedness of $A_j(\vr)$ is due to Lemmas \ref{L:Qbdd} and \ref{L:smooth-3}. In particular here we also used $\phi_j \big( X_*(\vr), r\beta_*(\vr), \vr \big)=0$ so the term involving the differentiation of $Q$ vanishes. The spectral stability of the steady $\omega_*$ corresponding to $\big( X_*(\vr), r\beta_*(\vr)\big)$ is determined by the spectrum $\sigma\big(A(\vr)\big)$ and the linear stability by $e^{t A(\vr)}$.

The linear operator $A(\vr)$ inherit an Hamiltonian structure from the nonlinear Hamiltonian one \eqref{E:VP-Ham-1}. Due to the complicated form of the symplectic operator $J_p$ in \eqref{E:VP-Ham-1}, instead of the variations 
\[
X = ( x_1, \ldots,  x_N), \quad \beta= ( \beta_1, \ldots, \beta_N), 
\]
we consider 
\begin{subequations} \label{E:conj-1}
\be \label{E:conj-1a} 
(X, \beta) = \BFQ(\vr) (y, \alpha), \quad (y, \alpha) = \big( (y_1^T, \ldots, y_N^T)^T,  (\alpha_1, \ldots, \alpha_N) \big) \in   \BFR^{2N} \times (Y^s)^N,
\ee
where 
\be \label{E:conj-1b} 
(x_j, \beta_j) = r_j Q\big(r_j, r_j \beta_{*j} (\vr)\big) f_j, \; \;  f_j = (\cos \theta, \sin \theta) y_j + \alpha_j,\;\; \alpha_j \in Y^s, \;\; j=1, \ldots, N. 
\ee
\end{subequations}
Here the splitting \eqref{E:Ys-1} of $\dot H^s$ is used. In this splitting we denote the corresponding projections  by
\[
\Pi_0: \big(\dot H^s (S^1) \big)^N \to \BFR^{2N}, \quad \Pi_j: \big(\dot H^s (S^1) \big)^N \to Y^s, \quad j=1, \ldots, N
\]
which result in the $y$ and $\alpha_j$ components. They also have bounded left inverses $\Pi_j^{-1}$, $0\le j \le N$. The following notation $Q_0 \in L(Y^s, X^s)$ defined based on Lemma \ref{L:Qbdd} will be used for convenience 
\be \label{E:Q0}
Q(1, 0) (\wt y_1 \cos \theta + \wt y_2 \sin \theta +  \wt \alpha) = (\wt y, Q_0 \wt \alpha), \quad \forall\ \wt y = (\wt y_1, \wt y_2)^T \in \BFR^2, \; \wt \alpha \in Y^s. 
\ee
According to Lemma \ref{L:Qbdd}, for any small $\vr$, $(X, \beta) \in \BFR^{2N} \times (X^s)^N $ and $(y, \alpha) \in \BFR^{2N} \times (Y^s)^N$ are isomorphic, so we shall consider the following conjugate operator of $A(\vr)$ 
\be \label{E:TAvr} \begin{split}
& \wt A(\vr) = \BFQ(\vr) ^{-1}A(\vr) \BFQ(\vr) =  (\wt A_0, \ldots, \wt A_N) (\vr) \\
&\wt A_j(\vr) (y, \alpha) = \Pi_j \CD_{(X, \beta)} ( r_1^{-1} \phi_1, \ldots, r_N^{-1} \phi_N)_{( X_*, M_{\vr} \beta_*, \vr)} \BFQ (y, \alpha).
\end{split} \ee
Based on the analysis in Section \ref{S:VP}, we shall first obtain the decomposition in Theorem \ref{T:spectrum} of $\BFR^{2N} \times \big(Y^s(S^1)\big)^N$ invariant under $\wt A(\vr)$. In the following lemma we prove that the principle part of $\wt A(\vr)$ is given by  
\be \label{E:B0} \begin{split}
&\big(B_0 (\vr)\big)_{2N \times 2N}  = \Lambda^{-1} J_N D^2 H_{\vec{\CC}} \big( X_*(\vr) \big), \\
& B_Y (\vr) = diag \big( \mu_1 r_1^{-2} \CD h(0) Q_0, \ldots,  \mu_N r_N^{-2} \CD h(0) Q_0\big) \\
&\qquad \;\,\; =\Lambda M_{\vr}^{-2}\CD h(0) Q_0 \in L\big( (Y^s)^N, (Y^{s-1})^N \big), 
\end{split} \ee
where $\Lambda$ and $J_N$ are defined in \eqref{E:symplectic-1} and $h$ in \eqref{E:h-1}. With a slight abuse of notations, when applied to $\BFR^{2N}$,  $M_{\vr}$ in the following also denotes
\be \label{E:Mr-1}
(M_{\vr})_{2N \times 2N} = diag (r_1 I_{2\times 2}, \ldots, r_N I_{2\times 2}).
\ee 

\begin{lemma} \label{L:TAr-1}
For any $s$, there exists $C>0$ such that, in the decomposition $\BFR^{2N} \oplus (Y^s)^N$, $\wt A(\vr)$ takes the form of 
\[
\wt A (\vr) = \begin{pmatrix} B_0 (\vr) + \wt B_0 (\vr)  & \wt B_{0Y} (\vr) \\ 
\wt B_{Y0} (\vr) & \wt B_Y (\vr)+ B_Y (\vr) \end{pmatrix} \in L\big(\BFR^{2N} \oplus (Y^s)^N, \BFR^{2N} \oplus (Y^{s-1})^N\big)
\]
and the corresponding norms satisfy 
\begin{align*}
& |\wt B_0 |+ |\wt B_{0Y} M_{\vr} |  + |\wt B_Y M_{\vr}| + | M_{\vr} \wt B_{Y0}|  \le C |\vr|.
\end{align*}
\end{lemma}

\begin{proof}
For any $(y, \alpha) \in  \BFR^{2N} \times (Y^s)^N$, let 
$(X, \beta) = \BFQ(\vr) (y, \alpha)$
as defined in \eqref{E:conj-1}. From Lemma \ref{L:Qbdd}, we have the more explicit form  
\[
(x_j, r_j \beta_j) = Q\big(1, r_j \beta_{*j} (\vr) \big)f_j, 
\] 
where $f_j$ is defined in \eqref{E:conj-1} as well. Since $|\beta_*|_{H^s} = O(|\vr|)$ proved in Theorem \ref{T:VP}, we have  from \eqref{E:Q0} and Lemma \ref{L:Qbdd}, 
\be \label{E:temp-1} 
|x_j - y_j| + | r_j \beta_j - Q_0 \alpha_j |_{H^s} \le C |r_j| |\vr| (|y_j| + |\alpha_j|_{H^s}) 
\ee
for some $C>0$ independent of $\vr$. 
From Lemma \ref{L:smooth-3}, 
\begin{align*}
&\CD_{(X, \beta)} \phi_j (X_*, M_{\vr} \beta_*, \vr) \BFQ (y, \alpha) 
=\mu_j  \CD h (r_j \beta_{*j}) \beta_j + r_j \CD_{(X, \beta)} R_j (X_*, M_{\vr} \beta_*, \vr) (X, \beta). 
\end{align*}
Since $|\beta_*|_{H^s} = O(|\vr|)$, 
from \eqref{E:temp-1} and Lemmas \ref{L:smooth-2} and \ref{L:smooth-3} we have 
\[
|\CD h (r_j \beta_{*j}) \beta_j - \CD h(0) \beta_j|_{H^{s-1}} 
\le C |\vr| ( |\alpha_j|_{H^s} + |r_j| |\vr| |y_j|), 
\]
\begin{align*}
& |\CD_{(X, \beta)} R_j (X_*, M_{\vr} \beta_*, \vr) (X, \beta) - \CD_{(X, \beta)} R_j (X_*, M_{\vr} \beta_*, 0) (y, 0)|_{H^{s-1}}\\
\le & C \big( |\vr| ( |X| + |\beta|_{H^{s}}) + |X-y| + |\nabla H_{\vec{\CC}} \big((X_*(\vr)\big)| |\beta|_{H^s}\big) \le  C |\vr| (|y| + |M_{\vr}^{-1} \alpha|_{H^s}),
\end{align*}
where the assumption $\nabla H_{\vec{\CC}} \big((X_*(0)\big)=0$ is used. They imply 
\begin{align*}
&|r_j^{-1} \CD_{(X, \beta)} \phi_j (X_*, M_{\vr} \beta_*, \vr) \BFQ (y, \alpha) - \mu_j r_j^{-1} \CD h(0) \beta_j -\CD_{(X, \beta)} R_j (X_*, 0, 0) (y, 0)|_{H^{s-1}} \\
\le & C |\vr| (|y| + |M_{\vr}^{-1} \alpha|_{H^s}). 
\end{align*}
While Lemma \ref{L:smooth-3} implies 
\[
\CD_{(X, \beta)} R_j (X_*, 0, 0) (y, 0) = \Lambda^{-1}J_N D^2 H_{\vec{\CC}} \big(X_*(\vr)\big) y \perp Y^{s-1}, 
\]
according to Lemma \ref{L:smooth-2}, we have,  for $\beta_j \in X^s$ 
\[
\CD h(0) \beta_j \in Y^{s-1}, \quad |\CD h(0) \beta_j -r_j^{-1}  \CD h(0) Q_0 \alpha_j|_{H^s} \le C |\vr| (|y_j| + |\alpha_j|_{H^s}).
\]
With the above inequalities, we obtain the desired estimates. 
\end{proof} 

The following decomposition lemma is essentially a spectral decomposition. 

\begin{lemma} \label{L:decom-1}
For any $s$, there exist $C>0$, $S_0 (\vr) \in L\big( \BFR^{2N}, (Y^s)^N\big)$, and $S_{Y^s} (\vr) \in L\big((Y^s)^N, \BFR^{2N} \big)$, such that $S_0$ is independent of $s$ and for any $s'>s$, 
\[
S_{Y^{s'}} = S_{Y^s}|_{(Y^{s'})^N}, \quad |M_{\vr}^{-1} S_0|+ |S_{Y^s} M_{\vr}^{-1} | \le C |\vr|, \quad \wt A(\vr) Z_0 \subset Z_0, \quad \wt A(\vr) \big( Z_{Y^s} \cap (H^{s+1})^N \big) \subset Z_{Y^{s}},  
\]
where 
\[
Z_{0, Y^s} = graph (S_{0, Y^s}) \subset \BFR^{2N} \oplus (Y^s)^N.
\]
Moreover, $S_{0}$ and $S_{Y^s}$ are the unique operators satisfying this invariance and 
\[
C |\vr|^3 |M_{\vr}^{-1} S_0| \le 1, \quad C |\vr|^3 |S_{Y^s} M_{\vr}^{-1}| \le 1. 
\]
\end{lemma}

\begin{proof} 
The invariance of the desired $S_0$ is equivalent to 
\begin{align*}
& \qquad \wt B_{Y0} + (\wt B_Y + B_Y) S_0 = S_0 ( B_0 + \wt B_0 + \wt B_{0Y} S_0) \Longleftrightarrow \\
& M_{\vr}^{-1} S_0 = B_Y^{-1} (M_{\vr}^{-1} S_0) \big( B_0 + \wt B_0 + \wt B_{0Y} M_{\vr} (M_{\vr}^{-1}S_0)\big) \\
& \qquad \qquad \qquad - M_{\vr}^{-1} B_Y^{-1} \big( M_{\vr}^{-1} M_{\vr} \wt B_{Y0} + \wt B_Y M_{\vr} (M_{\vr}^{-1}S_0) \big) 
\end{align*}
where $B_Y M_{\vr} = M_{\vr} B_Y$ was used. Since $B_Y \in L\big((Y^s)^N, (Y^{s-1})^N\big)$ is an isomorphism, the estimates in Lemma \ref{L:TAr-1}, which hold for all $s$, imply that the above right side is a contraction of $ S = M_{\vr}^{-1} S_0$ on the set $C |\vr|^3 |S| \le 1$ for some $C>0$. Therefore there exists a unique solution $S_0$ to the above equation such that $|S|\le C |M_{\vr}^{-1} B_Y^{-1} \wt B_{Y0}| \le C|\vr|$. In addition, for any $s>s'$, this $S_0 \in L\big( \BFR^{2N}, (Y^s)^N\big) \subset L\big( \BFR^{2N}, (Y^{s'})^N\big)$ is also the unique fixed point of the above operator equation on $L\big( \BFR^{2N}, (Y^{s'})^N\big)$. Therefore $S_0$ is independent of $s$. 

Similarly, the invariance property of $S_{Y^s}$ is equivalent to 
\begin{align*}
& (B_0 +\wt B_0) S_{Y^s}+ \wt B_{0Y} = S_{Y^s} ( \wt B_{Y0} S_{Y^s} + \wt B_Y + B_Y)\in L\big( (Y^{s+1})^N, \BFR^{2N} \big)\Longleftrightarrow \\
S_{Y^s} M_{\vr}^{-1} =&(B_0 +\wt B_0 - S_{Y^s}M_{\vr}^{-1}  M_{\vr} \wt B_{Y0} ) S_{Y^s} M_{\vr}^{-1} B_Y^{-1} + (\wt B_{0Y} M_{\vr} -   S_{Y^s} M_{\vr}^{-1}  M_{\vr} \wt B_Y M_{\vr})  M_{\vr}^{-2} B_Y^{-1}.
\end{align*}
Much as in the previous case, there exists a solution $S_Y$to the above equation, unique in the class $C |\vr|^3 |S_{Y^s}  M_{\vr}^{-1}| \le 1$,  such that $|S_{Y^s} M_{\vr}^{-1}|\le C |\wt B_{0Y} M_{\vr}^{-1} B_Y^{-1}| \le C|\vr|$. 
\end{proof}

Due to the smallness of $S_0$ and $S_{Y^s}$, clearly we have the decomposition $\BFR^{2N} \oplus (Y^s)^N = Z_0 \oplus Z_{Y^s}$ invariant under $\wt A(\vr)$. Let 
\be \label{E:TAr0}
\wt A^0 (\vr) = \Pi_0 \wt A (I+S_0) \in L(\BFR^{2N}) , \;\; \wt A^Y (\vr) = (I-\Pi_0) \wt A (I+S_{Y^s}) \in L\big( (Y^s)^N, (Y^{s-1})^N \big)
\ee
where the notation $\Pi_0$ originally defined in $L\big( (\dot H^s(S^1))^N, \BFR^{2N}\big)$ is slightly abused to denote the equivalent projection from $ \BFR^{2N} \oplus (Y^s)^N$ to $\BFR^{2N}$. Clearly $\wt A^{0, Y}(\vr)$ are actually $\wt A(\vr)$ restricted on $Z_0$ and $Z_{Y^s}$ which are parametrized by $\BFR^{2N}$ and $(Y^s)^{2N}$, respectively. Using the blockwise decomposition of $\wt A$ as  in Lemma \ref{L:TAr-1}, it is straight forward to compute 
\be \label{E:TAr0-1}
\wt A^0 = B_0 + \wt B_0 + \wt B_{0Y} S_0, \quad \wt A^Y = \wt B_{Y0} S_{Y^s} + \wt B_Y + B_Y. 
\ee
From Lemma \ref{L:TAr-1}, we obtain 
\be \label{E:TAr0-2}
|\wt A^0 - B_0 | \le C|\vr|.
\ee

In this invariant decomposition, $Z_0$ corresponds to small spectrum of $\wt A(\vr)$, where it behaves as a small perturbation of the $2N$-dim linearized point vortex dynamics $B_0$. Therefore if $B_0$ has unstable eigenvalues with positive real parts, then the linearized vortex patch dynamics $A(\vr)$ at $\big(X_*(\vr), \beta(\vr)\big)$ also has unstable eigenvalues. However, the spectral estimate based only on the above lemmas are far from optimal, particular along the $Z_{Y^s}$ subspace, until the Hamiltonian structure is incorporated. In fact, we can write 
\[
\wt A(\vr) = \BFQ(\vr) ^{-1}A(\vr) \BFQ(\vr) = \BFJ  \BFL (\vr)
\]
where for $j=0, \ldots, N$, 
\be \begin{split} \label{E:TAvr-1} 
& \BFJ= diag \big(\BFJ_0  , \frac{1}{\mu_1} \p_\theta, \ldots, \frac{1}{ \mu_N}  \p_\theta\big), \;\; (\BFJ_0)_{2N \times 2N} =- \Lambda^{-1} J_N
, \;\;
\BFL (\vr) = \big( L_0, \ldots, L_N) (\vr), \\
& L_j (\vr) (y, \alpha) =  \Pi_j \p_\theta^{-1} \CD_{(X, \beta)}(\frac {\mu_1}{r_1} \phi_1, \ldots, \frac {\mu_N}{r_N} \phi_N)|_{\big(X_*, \vr \beta_*, \vr\big)} \BFQ (y, \alpha). 
\end{split} \ee
Here $\p_\theta^{-1}: \dot H^{s-1} (S^1) \to \dot H^s(S^1)$ is an isomorphism. 
Clearly from Lemmas \ref{L:Qbdd}, \ref{L:smooth-2}, and \ref{L:smooth-3} we have 
\be \label{E:L-1}
\BFL \in C^\infty \big(D_R, L(\BFR^{2N} \times (Y^s)^N, \BFR^{2N}  \times (Y^s)^N \big), \quad \forall s,
\ee
while $\BFQ$ in $\BFL$ induces singularity as $\vr \to 0$ as indicated in Lemma \ref{L:Qbdd}, however. We have. 

\begin{lemma} \label{L:L-1}
For any $(y, \alpha) \in \BFR^{2N} \times (Y^s)^N$, $\BFL(\vr)$ satisfies 
\[
L_k (\vr) (y, \alpha) = \Pi_k (F_1, \ldots, F_N), \quad k=0, \ldots, N,
\]
where for $j=1, \ldots, N$, 
\begin{align*}
F_j (\theta) =&\mu_j \big( |\p_z \Gamma_{*j}|^{-1} (\nabla_{N_{*j}} \Psi_* ) \circ \Gamma_{*j}  \big) (e^{i\theta}) f_j(\theta) \\
&- \sum_{j'=1}^N \frac {\mu_j \mu_{j'}}{\pi r_j r_{j'}} \int_{S^1} \BFG_0\big(\Gamma_{*j} (e^{i\theta}) , \Gamma_{*j'} (e^{i\theta'})\big)  f_{j'}(\theta') d\theta'.
\end{align*}
Here $N_{*j}$ is the unit outward normal vector of $\p \Omega_{*j}$ and 
$f=(f_1, \ldots, f_N)
\in \big(\dot H^s(S^1)\big)^N$ is defined as in \eqref{E:conj-1}.
Moreover, we have 
\be \label{E:comm}
\p_\theta L_j(\vr) (0, \cdot) - L_j (\vr) (0, \p_\theta \ \cdot) \in L\big( (Y^s)^N, Y^s\big), \quad j=1, \ldots, N,
\ee
and, for any $(y, \alpha), (\wt y, \wt \alpha) \in \BFR^{2N} \times (Y^s)^N$, 
\[
\langle \BFL(\vr) (y, \alpha), (\wt y, \wt \alpha)\rangle = \langle \BFL(\vr) (\wt y, \wt \alpha), (y, \alpha)\rangle, 
\]
where  $\langle \cdot, \cdot\rangle$ the $L^2$ duality pair. 
\end{lemma}

\begin{proof}
To calculate more explicitly the differentiation of $\phi_j$ defined in \eqref{E:phi-VP-1}, for any $(y, \alpha) \in  \BFR^{2N} \times (Y^s)^N$, let $(X, \beta) = \BFQ(\vr) (y, \alpha)$ as defined in \eqref{E:conj-1}  
and $\wt \Gamma_j(\ep)$, $\omega(\ep)$, $\Gamma_j (\ep)$, $\Psi(\ep)$, $\phi_j(\ep)$  be defined by $(X_*(\vr) + \ep X, \vr \beta_* (\vr) + \ep \beta)$. For simplicity of notations, we skip the $\ep$ in the next a few lines, as well as the $\vr$ dependence of $(X_*,\beta_*)$. We have  
\begin{align*}
&\Big(\CD_{(X, \beta)}\phi_j|_{\big(X_*, \vr \beta_*, \vr\big)} \BFQ (y, \alpha) \Big)(\theta) = \p_\theta \frac d{d\ep} \Big( \Psi\big( \Gamma_j (e^{i\theta}) \big)\Big)|_{\ep=0}\\
=& \p_\theta \big( (\p_\ep \Psi)|_{\ep=0} \circ \Gamma_{*j} + (\p_\ep \Gamma_j)|_{\ep=0} \cdot (\nabla \Psi_*) \circ \Gamma_{*j} \big) (e^{i\theta})\\
=& \p_\theta \big( (\p_\ep \Psi)|_{\ep=0} \circ \Gamma_{*j} +  r_j f_j |\p_z \Gamma_{*j}|^{-1}(\nabla_{N_{*j}} \Psi_* ) \circ \Gamma_{*j} \big) (e^{i\theta}),
\end{align*}
where in the last step we used $\Psi_* |_{\p \Omega_{*j} }\equiv const$ and the fact that $r_j f_j |\p_z \Gamma_{*j}|^{-1} $ is equal to the normal component $\nu_j$ of $(\p_\ep \Gamma_j)|_{\ep=0}$ according to the definition of $Q$ in Lemma \ref{L:Qbdd}. Using Lemma \ref{L:Stream-F}, one may compute $(\p_\ep \Psi)|_{\ep=0}$ 
\[
(\p_\ep \Psi)|_{\ep=0} (x) = - \int_\Omega \BFG_0(x, y) (\p_\ep \omega )|_{\ep=0} (y) d\mu_y. 
\]
According to \eqref{E:tangent-p}, $\p_\ep \omega$ is a measure supported in $\p \Omega_j$, 
\[
(\p_\ep \omega )|_{\ep=0} = \sum_{k=1}^N \frac {\mu_k}{\pi r_k^2} \nu_k \delta_{\p \Omega_{*j}},
\]
which along with $\nu_k \circ \Gamma_{*k} =r_k f_k |\p_z \Gamma_{*k}|^{-1} $ implies 
\[
(\p_\ep \Psi)|_{\ep=0} (x) = - \sum_{k=1}^N \frac {\mu_k}{\pi r_k} \int_{S^1} \BFG_0\big(x, \Gamma_{*k} (e^{i\theta'})\big)  f_k(\theta') d\theta'. 
\]
Therefore, 
\begin{align*}
\Big(\CD_{(X, \beta)}\phi_j|_{\big(X_*, \vr \beta_*, \vr\big)} \BFQ (y, \alpha) \Big)(\theta) = \p_\theta \Big(&- \sum_{j'=1}^N \frac {\mu_{j'}}{\pi r_{j'}} \int_{S^1} \BFG_0\big(\Gamma_{*j} (e^{i\theta}), \Gamma_{*j'} (e^{i\theta'})\big)  f_{j'}(\theta') d\theta' \\
& + r_j \big( |\p_z \Gamma_{*j}|^{-1} (\nabla_{N_{*j}} \Psi_* ) \circ \Gamma_{*j}  \big) (e^{i\theta}) f_j(\theta) \Big), 
\end{align*}
which and $\Pi_k \p_\theta^{-1} \p_\theta = \Pi_k$ immediately yield the formula for $L_k(\vr)$ for $0\le k\le N$. 

For any $(y, \alpha), (\wt y, \wt \alpha) \in \BFR^{2N} \times (Y^s)^N$, let $f, \wt f\in \big(\dot H^s(S^1)\big)^N$ be defined as in \eqref{E:conj-1}. From the definition of $\Pi_k$, $k=0, \ldots,, N$, we have 
\begin{align*}
\langle \BFL(\vr) (y, \alpha), (\wt y, \wt \alpha)\rangle = \Sigma_{j=1}^N \langle F_j, \wt f_j\rangle =& \sum_{j=1}^N\int_{S^1}  \mu_j \Big( |\p_z \Gamma_{*j}|^{-1} \big(\nabla_{N_{*j}} \Psi_*   \big) \circ \Gamma_{*j}  \Big) (e^{i\theta}) f_j(\theta) \wt f_j (\theta) d\theta \\
&-\sum_{j, j'=1}^N \frac {\mu_j \mu_{j'}}{\pi r_j r_{j'}} \int_{S^1\times S^1}  \BFG_0\big(\Gamma_{*j} (e^{i\theta}) , \Gamma_{*j'} (e^{i\theta'})\big)  f_{j'}(\theta') \wt f_j(\theta) d\theta' d\theta
\end{align*}
and thus the symmetry of $\BFL(\vr)$ follows. 

To complete the proof of the lemma, we show the  estimate on the commutator $[\p_\theta, L_j]$ for $1\le j \le N$. For any $\alpha \in (Y^s)^N$, let $(f_1, \ldots, f_N)$ be defined as in \eqref{E:conj-1} by $(0, \alpha)$ and 
\[
\wt L_j \alpha = \Pi_j (\wt F_1, \ldots, \wt F_N), 
\quad \wt F_k (\theta) = \frac {\mu_k^2}{\pi r_k^2} \int_{S^1} \log |\Gamma_{*k} (e^{i\theta}) - \Gamma_{*k} (e^{i\theta'})| f_k (\theta') d\theta'.
\]
From the definition \eqref{E:Gc} of $\BFG_0$, it is straight forward to prove that 
\[
\alpha \longrightarrow \p_\theta \big( L_j (0, \alpha) - \wt L_j \alpha\big) - \big( L_j (0, \p_\theta \alpha) - \wt L_j \p_\theta \alpha \big) \in L\big( (Y^s)^N, Y^s\big). 
\]
With slight abuse of the notation $\Pi_j$, one may compute through integration by parts 
\begin{align*}
\wt L_j \p_\theta \alpha = & \frac {\mu_j^2}{\pi r_j^2} \Pi_j \int_{S^1} \log |\Gamma_{*k} (e^{i\theta}) - \Gamma_{*k} (e^{i\theta'})| \p_\theta f_j (\theta') d\theta'\\
= & \frac {\mu_j^2}{\pi r_j^2} \Pi_j \text{Re} \int_{S^1} \frac { i e^{i\theta'} \p_z \Gamma_{*j} (e^{i\theta'})}{\Gamma_{*j} (e^{i\theta}) - \Gamma_{*j} (e^{i\theta'})} f_j (\theta') d\theta'.
\end{align*}
Therefore 
\begin{align*}
\p_\theta \wt L_j \alpha - \wt L_j \p_\theta \alpha = \frac {\mu_j^2}{\pi r_j^2} \Pi_j \text{Re} \int_{S^1} \frac {i e^{i\theta} \p_z \Gamma_{*j} (e^{i\theta}) -i e^{i\theta'} \p_z \Gamma_{*j} (e^{i\theta'})}{\Gamma_{*j} (e^{i\theta}) - \Gamma_{*j} (e^{i\theta'})} f_j (\theta') d\theta'.
\end{align*}
Since the above kernel $\frac { i e^{i\theta} \p_z \Gamma_{*j} (e^{i\theta}) -i e^{i\theta'} \p_z \Gamma_{*j} (e^{i\theta'}) }{\Gamma_{*j} (e^{i\theta}) - \Gamma_{*j} (e^{i\theta'})}$ is a smooth function of $\theta$ and $\theta'$, we obtain the boundedness of above commutator. This completes the proof of the lemma. 
\end{proof}

The above lemma implies that the linearized vortex patch dynamics 
possesses a Hamiltonian structure. We shall further analyze $\BFL(\vr)$ on $L^2$ based on Lemmas \ref{L:TAr-1} and \ref{L:decom-1}. 

\begin{lemma} \label{L:L-2}
There exists $C>0$ such that, for any $(y, \alpha)\in \BFR^{2N} \times (Y^0)^N$,
\begin{align*}
&|\langle \BFL (y, \alpha), (y, \alpha)\rangle + \langle D^2 H_{\vec{\CC}} (X_*)y, y\rangle - \sum_{j=1}^N \sum_{k=2}^\infty  \frac {\mu_j^2}{r_j^2} \frac {k-1}{2k \pi} |\hat \alpha_j (k)|^2| \\
\le & C  |\vr| ( |y|^2 + |y| |M_{\vr}^{-1} \alpha|_{L^2} + |\alpha|_{L^2}  |M_{\vr}^{-1} \alpha|_{L^2}),
\end{align*}
where $\hat \alpha_j(k)$ is the $k$-th Fourier coefficient of $\alpha_j$. 
\end{lemma}

\begin{proof}
Since the symmetric $\BFL (\vr) = \BFJ^{-1} \wt A (\vr)$, it can be written as 
\[
\BFL (\vr) = \begin{pmatrix} \Lambda J_N (B_0 + \wt B_0)   & \Lambda J_N  \wt B_{0Y} \\ 
\Lambda \p_\theta^{-1} \wt B_{Y0} & \Lambda \p_\theta^{-1} (\wt B_Y + B_Y) \end{pmatrix} \in L\big(\BFR^{2N} \oplus (Y^s)^N \big).
\]
From Lemma \ref{L:TAr-1}, we have 
\begin{align*}
& |\langle \BFL (y, \alpha), (y, \alpha)\rangle - \langle \Lambda J_N B_0y, y\rangle - \langle \Lambda \p_\theta^{-1} B_Y \alpha, \alpha\rangle| \\
= & | \langle \Lambda J_N \wt B_0y, y\rangle + \langle \Lambda J_N  \wt B_{0Y} M_{\vr} M_{\vr}^{-1} \alpha, y \rangle + \langle \Lambda \p_\theta^{-1} M_{\vr} \wt B_{Y0} y, M_{\vr}^{-1}\alpha \rangle + \langle \Lambda \p_\theta^{-1} \wt B_Y M_{\vr} M_{\vr}^{-1} \alpha, \alpha \rangle | \\
\le & C |\vr| ( |y|^2 + |y| |M_{\vr}^{-1} \alpha|_{L^2} + |\alpha|_{L^2}  |M_{\vr}^{-1} \alpha|_{L^2}).
\end{align*}
From \eqref{E:B0}, we have 
\be \label{E:B0-1}
\Lambda J_N B_0 = - D^2 H_{\vec{\CC}} \big(X_*(\vr)\big), \quad \Lambda \p_\theta^{-1} B_Y = \Lambda^2 M_{\vr}^{-2} \p_\theta^{-1} \CD h(0) Q_0. 
\ee
Direct computations based on \eqref{E:Q0} and Lemma \ref{L:Qbdd} and \ref{L:smooth-2} yield 
\be \label{E:B0-2}
\p_\theta^{-1} \CD h(0) Q_0 \sum_{k=2}^\infty (A_k \cos k \theta + B_k \sin k\theta) =  \sum_{k=2}^\infty \frac {k-1}{2k\pi} (A_k \cos k \theta + B_k \sin k\theta)
\ee
and the desired estimate follows. 
\end{proof}

Along with Lemma \ref{L:decom-1}, we immediately obtain the uniform positivity of $\BFL$ on $Z_{Y^0}$. 

\begin{corollary} \label{C:posi}
There exists $\delta>0$ such that, for any $(y, \alpha)\in Z_{Y^0}$, 
\[
\langle \BFL (y, \alpha), (y, \alpha)\rangle \ge \delta | M_{\vr}^{-1} \alpha|_{L^2}^2. 
\]
\end{corollary}

Consequently $\BFL$ has a finite Morse index for $s=0$, namely, on $\BFR^{2N} \times (Y^0)^{2N}$. In fact, the quadratic form $\langle \BFL\cdot, \cdot \rangle$ has only finitely many non-positive directions. There have been extensive results on such linear Hamiltonian operators $\wt A= \BFJ \BFL$, e.g. see in \cite{LZ17}. In particular we have that the $\BFL$-orthogonal complement of $Z_{Y^0}$ in $\BFR^{2N} \times Y^0$ is also invariant under $\wt A$. Due to $Z_{Y^s}$ being close to $\{y=0\}$ and the uniqueness of invariant subspaces near $\{\alpha=0\}$ given in Lemma \ref{L:decom-1},  $Z_0$ must be the $\BFL$-orthogonal complement of $Z_{Y^0}$.

\begin{corollary} 
It holds that $\langle \BFL (y, \alpha), (\wt y, \wt \alpha)\rangle =0$ for any $(y, \alpha) \in Z_0$ and $(\wt y, \wt \alpha) \in Z_{Y^0}$. 
\end{corollary}

According to Proposition 12.1 and Lemma 3.1 in \cite{LZ17}, $\wt A$ generates a $C^0$ group $e^{t\wt A}$ on $\BFR^{2N} \times (Y^0)^N$ such that 
\begin{enumerate} 
\item $e^{t\wt A}$  preserves $\langle \BFL \cdot, \cdot \rangle$ and $\langle \BFJ^{-1}\cdot, \cdot\rangle$, namely $(e^{t\wt A})^* \BFL e^{t\wt A} =\BFL$ and $(e^{t\wt A})^* \BFJ^{-1} e^{t\wt A} =\BFJ^{-1}$, where $(\cdot)^*$ denote the adjoint operator on $L^2$; 
\item $e^{t\wt A} (Z_0) = Z_0$ and $e^{t\wt A} (Z_{Y^0}) = Z_{Y^0}$. 
\end{enumerate}

Furthermore we have 

\begin{lemma} \label{L:stability-ZY} 
There exists $C>0$ such that 
\begin{enumerate}
\item $|e^{t\wt A}|_{Z_{Y^0}}| \le C, \quad \forall \  t\in \BFR$.
\item Consequently, $\sigma ( \wt A|_{Z_{Y^0}}) \subset \{ \lambda \in i\BFR\mid |\lambda| \ge C^{-1} |\vr|^{-2}\}$.
\item $\sigma (A) = \sigma(\wt A) = \sigma ( \wt A|_{Z_0} ) \cup \sigma ( \wt A|_{Z_{Y^0}})$ and $|\lambda| \le C$ for any $\lambda \in \sigma ( \wt A|_{Z_0} )$. 
\end{enumerate}
\end{lemma}

When the spectra are concerned in the above, $A$, $\wt A$, and $\wt A|_{Z_{Y^0}}$ are considered as (unbounded) closed operators on $\BFR^{2N}\times (X^0)^N$, $\BFR^{2N}\times (Y^0)^N$, and $Z_{Y^0}$ respectively. 

\begin{proof}
Conclusion (3) is a direct consequence of the invariance of $Z_0$ and $Z_{Y^0}$ under $\wt A$ and \eqref{E:TAr0-2}, \eqref{E:B0}. Conclusion (1) also follows from the invariance of $Z_{Y^0}$, the conservation of $\langle \BFL\cdot, \cdot \rangle$, and its positivity on $Z_{Y^0}$ (Corollary \ref{C:posi}), which also implies that $\sigma ( \wt A|_{Z_{Y^0}})$ is purely imaginary. To complete the proof of the lemma, we only need to show the lower bound on spectral points of $\wt A|_{Z_{Y^0}}$. In fact, \eqref{E:B0-1} and \eqref{E:B0-2} imply that $B_Y$ is anti-self-adjoint on $(Y^0)^N$ and there exists $\delta >0$ such that 
\[
|M_{\vr}^{-2}(i\lambda - B_Y)^{-1}  |_{L\big( (Y^0)^N\big)} \le C, \quad \forall \lambda \in \BFR \backslash [-\delta |\vr|^2, \delta |\vr|^2]. 
\]
From \eqref{E:TAr0-1} and $B_Y M_{\vr} = M_{\vr} B_Y$, we have 
\[
i\lambda - M_{\vr} \wt A^Y M_{\vr}^{-1} =\big(I- (M_{\vr} \wt B_{Y0} S_Y M_{\vr}^{-1} + M_{\vr} \wt B_Y M_{\vr} M_{\vr}^{-2}) (i\lambda - B_Y)^{-1}    \big) (i\lambda - B_Y). 
\]
Lemmas \ref{L:TAr-1} and \ref{L:decom-1} and the above resolvent inequality imply 
\[
|(M_{\vr} \wt B_{Y0} S_Y M_{\vr}^{-1} + M_{\vr} \wt B_Y M_{\vr} M_{\vr}^{-2}) (i\lambda - B_Y)^{-1} | \le C |\vr| |M_{\vr}|
\]
and thus we obtain $i\lambda \in \sigma (M_{\vr} \wt A^Y M_{\vr}^{-1}) = \sigma ( \wt A|_{Z_{Y^0}})$. \end{proof}

\begin{remark} \label{R:stability}
The above results imply that, in the $L^2$ metric on $(X, \beta)$, the linearized vortex patch dynamics $e^{tA(\vr)}$ at $(X_*(\vr), \beta_*(\vr)\big)$ is spectrally and linearly stable in the $2N$-codim invariant subspace $\BFQ Z_{Y^0}$ and thus its stability is completely determined by the flow restricted on the $2N$-dim invariant subspace $\BFQ Z_0$. This $2N$-dim linear dynamics $e^{tA(\vr)}|_{\BFQ Z_0}$ is conjugate to $e^{t\wt A^0}$ on $Z_0$ which is a perturbation to $e^{tB_0}$, the linearized point vortex dynamics. In fact, parametrizing $Z_0$ by $\{ \alpha =0\}$, $e^{t\wt A^0}|_{Z_0}$ also has  a Hamiltonian  
\[
\wt H (\vr, y) = \frac 12 \langle (I+S_0)^* \BFL (I+S_0)y, y\rangle, \quad y\in \BFR^{2N},
\]
and an appropriate symplectic structure. Therefore all the  perturbation results on finite dimensional linear Hamiltonian systems apply, including, but not limited to  
\begin{enumerate}
\item If $\pm D^2 H_{\vec{\CC}} \big(X_*(0)\big) >0$, then $e^{tA(\vr)}$ is stable. 
\item If $\sigma \big(B_0 (0) \big) \subset i\BFR \backslash \{0\}$ and, when restricted to the eigenspace of any eigenvalue of $B_0(0)$, $D^2 H_{\vec{\CC}} \big(X_*(0) \big)$ is either positive or negative, then $e^{tA(\vr)}$ is stable.This case is sometimes referred to as strongly linearly stable. 
\item If $\exists \lambda \in \sigma \big(B_0 (0) \big) \backslash  i\BFR$, then $e^{tA(\vr)}$ is unstable and has exponential trichotomy. 
\item If the nonsingular symmetric matrix  $B_0 (0)$ has an odd Morse index, namely, number of negative directions, then it has a positive eigenvalue and thus the above unstable case occurs. 
\end{enumerate}
See, for example, \cite{Ar89, LZ17} for more details for spectral and linear analysis. 
\end{remark}

To end this section, we consider the spectral of the linearized vortex patch dynamics at $(X_*(\vr), \beta_*(\vr)\big)$ in $\BFR^{2N} \times (X^s)$ for $s>0$. In fact we have 

\begin{lemma} \label{L:spectral-s}
When considered as unbounded closed operators on $\BFR^{2N} \times (X^s)^N$, $s\ge 0$, the spectrum $\sigma\big( A(\vr) \big)$ consists of isolated eigenvalues only and is independent of $s$.  
\end{lemma}

\begin{proof}
The proof is rather standard 
and we shall only sketch it. Let $\sigma_s$ denote the spectrum of $A(\vr)$ on $\BFR^{2N} \times (X^s)^N$. We shall work on the conjugate operator $\wt A(\vr)$. On the one hand, since $B_Y^{-1} \in L \big( (Y^s)^N, (Y^{s+1})^N\big)$ is compact when considered in $L \big( (Y^s)^N\big)$, $\sigma_s$ has only isolated eigenvalues for any $s$ and any eigenfunctions in $\BFR^{2N} \times (Y^0)^N$ automatically belongs to $\BFR^{2N} \times (Y^s)^N$, which implies $\sigma_0 \subset \sigma_s$. On the other hand, let $\lambda \notin \sigma_0$. Apparently $\lambda - \wt A$ is one-to-one on $ \in \BFR^{2N} \times (Y^s)^N$. First consider the case when $s>0$ is an integer. For any $v \in \BFR^{2N} \times (Y^s)^N$, since $\lambda \notin \sigma_0$, we first have $u=(\lambda -A)^{-1} v \in \BFR^{2N} \times (Y^1)^N$. Inductively the commutator estimate \eqref{E:comm} implies $u \in \BFR^{2N} \times (Y^{s+1})^N$. Therefore $\lambda -\wt A: \BFR^{2N} \times (Y^{s+1})^N \to \BFR^{2N} \times (Y^s)^N$ is an isomorphism and thus $\lambda \notin \sigma_s$, which also holds for general $s>0$ through interpolation. This completes the proof. 
\end{proof}

\begin{corollary}
$\wt A (Z_{Y^{s+1}}) = Z_{Y^s}$ for any $s\ge 0$.
\end{corollary}

Even though $A(\vr)$ has the same spectra on $\BFR^{2N} \times (X^s)^N$, $s\ge 0$, and $e^{t\wt A(\vr)}|_{Z_{Y^0}}$ is linearly stable, it might happen that it has some algebraic growth on $Z_{Y^s}$, $s>0$, in $H^s$ norm.

\section{Continuous concentrated steady vorticities} \label{S:C0}

In this section, fixing the vorticities $\mu_1, \ldots, \mu_N \in \BFR\backslash\{0\}$ as given in Theorem \ref{T:C0} and $s>\frac 32$, we shall prove Theorem \ref{T:C0} on the existence of $C^{0,1}$ concentrated steady vorticities $\omega = \Sigma_{j=1}^N \omega_j$ satisfying \eqref{E:localized-V} and \eqref{E:localization}. According to Lemma \ref{L:steady-C0}, $\omega$ is steady if, for each $j$, $\omega_j =\Delta \Psi= \wt f(\Psi)$ on $\Omega_j= supp(\omega_j)$ for some $\wt f$, where $\Psi$ is the stream function of $\omega$ given in Lemma \ref{L:Stream-F}. 
Since $\Omega_j$ is close to a disk, we start with the above semilinear elliptic equation on $B_1$. For each $j=1, \ldots, N$, we take $f_j$ such that 
\begin{subequations}  \label{localpatch:gammaassumptions} 
\be 
f_j \in C^\infty (\BFR, \BFR), \; \text{odd},
\qquad  f_j^\prime(0) < 0, \quad f_j (\tau) > 0 \textrm{ on } \BFR^-, 
\ee 
\be 
\exists \; \text{a negative radial solution } \psi_j^* \text{ to } \quad \Delta \psi_j = f_j (\psi_j)  \; \textrm{ in } \; B_1, \qquad \psi_j|_{S^1} =0, 
\label{E:f-2} \ee
and 
\be 
\Delta - f_j^\prime(\psi_j^*) \textrm{ is non-degenerate}. \label{localpatch:nondegen}
\ee \end{subequations}
More specifically, \eqref{localpatch:nondegen} requires that $\Delta - f_j^\prime(\psi_j^*)$, viewed as a self-adjoint operator on $L^2(B_1)$, does not contain $0$ in its spectrum; this is of course a generic condition.  Indeed, the class of $f_j$ satisfying \eqref{localpatch:gammaassumptions} is quite large (cf., e.g., \cite{brezis1997elliptic,figueiredo1982apriori,lions1982positive}). 
As in Section \ref{S:VP}, fix $\rho>0$ satisfying \eqref{E:rho-1} 
Our procedure to obtain Lipschitz concentrated steady vorticity is roughly the following:
\begin{itemize}
\item [{\bf Step 1.}] For any 
\[
X =(x_1, \ldots, x_N) \in \Sigma_\rho, \quad \beta = (\beta_1, \ldots, \beta_N) \in B_{(X^s)^N, R_s},
\quad \vec{r} \in D_R, 
\]
where $X^s$, $R_s$, $D_R$ are defined in \eqref{E:Xs}, \eqref{E:Rs}, and \eqref{E:D-R} and $B_{(X^s)^N, R_s }$ is the ball of radius $R_s$ in $(X^s)^N$, define $\wt \Gamma_j$, $\Gamma_j$, and $\Omega_j = \Gamma_j(B_1)$ as in \eqref{E:CM-2-1}, \eqref{E:CM-2-2}, and \eqref{E:a1j}. 
\item  [{\bf Step 2.}] Through the conformal mapping coordinates $\Gamma_j$, consider the elliptic problem  
\be \label{E:semiLE-1}
\Delta \wt \psi = |1+ \p_z \wt \Gamma_j|^{2} f_j (\wt \psi), \qquad  \wt \psi: B_1 \to \BFR, \quad \wt \psi|_{S^1}=0. 
\ee
By Implicit Function Theorem, there exists a unique solution $\wt \psi_j$ close to $\psi_j^*$ in $H^{s+\frac 32} (B_1)$ topology. Let 
\be \label{E:Vor-j-1}
\psi_j= \frac { \mu_j\wt \psi_j  \circ \Gamma_j^{-1} \chi_{\Gamma_j(B_1)} } {\int_{B_1} 
|1+ \p_z \wt \Gamma_j|^{2}f_j(\wt \psi_j) d\mu}  , \quad \omega_j = \frac {\mu_j \chi_{\Gamma_j(B_1)}f_j (\wt \psi_j  \circ \Gamma_j^{-1})}{ (a_1^jr_j)^2  \int_{B_1} |1+ \p_z \wt \Gamma_j|^{2} f_j(\wt \psi_j) d\mu}.
\ee
Clearly 
\be \label{E:Vor-j-2}
\omega_j =\Delta \psi_j \; \text{ on }\; \Omega_j = \Gamma_j(B_1) =supp(\omega_j), \quad \int_{\Omega_j} \omega_j d\mu=\mu_j.
\ee
\item  [{\bf Step 3.}] Define 
\be \label{E:Vor-1} 
\omega= \sum_{j=1}^N \omega_j, \;\; \text{ and }\; \Psi= (\Delta_0^{-1} \omega) + \sum_{j=1}^n c_j \CH_j \; \text{ as defined by Lemma \ref{L:Stream-F}}.
\ee 
By Lemma  \ref{L:Stream-F}, it holds that $\Delta \Psi = \Delta \psi_j$ on $\Omega_j$, $j=1, \ldots, N$. For $|\vec{r}| << \rho$, the domains $\Omega_j$ are mutually disjoint. So if $\Psi = const$ on each $\p \Omega_j$, then $\Psi-\psi_j=const$ on each $\Omega_j$ and thus by Lemma \ref{L:steady-C0},  $\omega$ is a steady solution to the Euler equation. Therefore consider 
\be \label{E:phi-1}
\phi_j (X, \beta, \vec{r}) (\theta) =\p_\theta \Psi\big( \Gamma_j(e^{i\theta}) \big) 
\ee
Equivalently, we have 

\begin{lemma} \label{L:steady-C0-1}
If $\phi_j (X, \beta, \vec{r}) \equiv 0$ for $j=, 1\ldots, N$, then the corresponding $\omega$ is a steady state of the Euler equation.
\end{lemma}

The above condition will be achieved by choosing proper $X \in \Sigma_\rho$ and $\beta \in (X^s)^N$ through a perturbation argument carried in the rest of the section. \\
\end{itemize}

As in Section \ref{S:VP}, we start with the regularity of  $\phi_j$ and its smoothness in $X$, $\beta$, and $\vec{r}$. 

\begin{lemma} \label{L:smooth-1}
There exists $R>0$ such that 
\[
\phi_j(X, \beta, \vec{r})= r_j R_j (X, \beta, \vec{r}) + \mu_j h_j(\beta_j) 
\]
with 
\[
\big(h_j (\beta_j)\big)(\theta)  =  \frac 1{2\pi} \p_\theta \int_{\Omega_j} \omega_j (y)\log |\Gamma_j (e^{i\theta}) - y|  d\mu_y
\]
satisfying 
\begin{align*}
&h_j \in C^\infty \big(B_{X^s, R_s}, \dot H^{s-1}(S^1)\big), \quad  R_j \in C^\infty \big(\Sigma_\rho  \times B_{(X^s)^N, R_s} \times \overline{D_R}, \dot H^{s-1}(S^1)\big),\\
&h_j(0) =0, \quad R_j(X, \beta, 0)=  \lim_{|\vec{r}| \to +0} R_j (X, \beta, \vec{r}) = 
- \mu_j^{-1} a_1^j\big(i e^{i\theta}(1+ \p_z \wt \Gamma_j (e^{i\theta}) )\big)\cdot  \nabla_{x_j} H_{\vec{\CC}} (X)
\end{align*}
where $H_{\vec{\CC}}$ was defined in \eqref{E:E0}. 
\end{lemma}

See Remark \ref{R:smooth} for some discussions on the regularity of $h_j(\beta_j)$. 


\begin{proof}
Much as the calculations in \eqref{E:h-1}, we have 
\begin{align*}
h_j (\beta_j) (\theta) 
=& \frac { \mu_j}{2\pi \int_{B_1} |1+ \p_z \wt \Gamma_j|^{2} f_j(\wt \psi_j) d\mu} \text{Re}\int_{B_1} |1+ \p_z \wt \Gamma_j|^{2} f_j(\wt \psi_j) \frac {  ie^{i\theta}  \big(1+ \p_z \wt \Gamma_j (e^{i\theta}) \big)}{ e^{i\theta} - z +\wt\Gamma_j  (e^{i\theta}) -\wt \Gamma_j(z)} d \mu_z. 
\end{align*}
It is readily clear that $h_j$ indeed depends only on $\beta_j$ and in particular independent of $\vr$. The property $h_j(\beta_j)\in \dot H^{s-1}(S^1)$ and its smooth dependence on  $\beta_j$ follow from Lemma B.3 in  \cite{SWZ13} on which the proof of Lemma \ref{L:smooth-2} was actually based. When $\beta_j=0$, $\Gamma_j(z)= x_j + r_j z$, $\wt \psi_j =\psi_j^*$, and $\Omega_j$ is a disk of radius $r_j$. Therefore $h_j(0)=0$ since $\omega_j$ is radially symmetric. 

The regularity of $R_j$ and its smooth dependence on $(X, \beta, \vr)$ can be proved much as in Lemma \ref{L:smooth-3}. We shall just point out the major modifications. 

By using \eqref{E:semiLE-1}, it is straight forward to prove that the mapping $\beta_j \to \wt \psi_j$ belongs to $C^\infty \big(X^s, H^{s+\frac 32} (B_1)\big)$ (see Lemma B.1 in \cite{SWZ13} for details). For any smooth function $\gamma(x, y)$ defined on $\Omega^2$ and $1\le j, j' \le N$, one may compute 
\begin{align*}
\p_\theta \int_{\Omega_{j'}} \gamma\big( \Gamma_j (e^{i\theta}), y\big) \omega_{j'} (y) d\mu_y =& \frac {a_1^j r_j \mu_{j'}}{\int_{B_1} |1+ \p_z \wt \Gamma_{j'}|^{2} f_{j'}(\wt \psi_{j'}) d\mu} \int_{B_1} |1+ \p_z \wt \Gamma_{j'}|^{2} f_{j'}(\wt \psi_{j'}) \times \\
&\qquad \quad  \big((\nabla_1 \gamma) \big(\Gamma_j (e^{i\theta}), \Gamma_{j'} (z)\big)\big) \cdot \big( ie^{i\theta} (1+ \p_z \wt \Gamma_j (e^{i\theta}) )\big) d\mu_z
\end{align*}
and thus such quantities are valued in $\dot H^{s-1}(S^1)$ and smooth in $(X, \beta, \vr)$. Therefore, the smoothness of $\tilde g(x, y)$ and $\CH_j (x)$ in the definition of $\Psi$, as well as the smoothness of $\log|x-y|$ for $x \in \Omega_j\ne \Omega_{j'} \ni y$, implies the smooth dependence of $R_j$ on $X$, $\beta$, and $\vr$. 
Finally, $\Gamma_j (z) =x_j$ if $r_j=0$,  the above expressions yields the formula of $R_j(X, \beta, 0)$ just as in Lemma \ref{L:smooth-3}.  
\end{proof}

Define $\CF (X, \beta, \vr) = (\CF_1, \ldots, \CF_N) (X, \beta, \vr)$ by exactly the same formula as in \eqref{E:CF-VP-1} and split $\big(\dot H^s(S^1)\big)^N = Y_1 \oplus (Y^s)^N$ orthogonally with respect to $L^2 (S^1)$ as in \eqref{E:Ys-1}.  
The above Lemma implies the following lemma parallel to Lemma \ref{L:DCF-2}. 

\begin{corollary} \label{C:CF}
$\CF \in C^\infty \big(\Sigma_\rho  \times B_{(X^s)^N, R_s} \times \overline{D_R}, \dot H^{s-1}(S^1)\big)$. Moreover, for any $X \in \Sigma_\rho$, $\CD_\beta \CF (X, \beta, 0)$ is isomorphic from $(X^s)^N$ to $(Y^{s-1})^N$ and $\CD_X \CF(X, \beta, 0)(\BFR^{2N})\subset Y_1$ and  
\be \label{E:DCF-2} \begin{split}
&\CD_\beta \CF (X, \beta, 0) \alpha = (\mu_1 \CD h_1(0) \alpha_1, \ldots,  \mu_N\CD h_N(0) \alpha_N),  \\
& \CD_X \CF(X,\beta, 0) \wt X = \big( (\cos \theta, \sin \theta) y_1, \ldots, (\cos \theta, \sin \theta) y_N \big), 
\end{split}\ee
where 
\[
y = (y_1^T, \ldots, y_N^T)^T = \Lambda^{-1} J_N D^2 H_{\vec{\CC}} (X) \wt X
\]
and $\Lambda$ and $J_N$ are defined in \eqref{E:symplectic-1}. 
\end{corollary}

Unlike in Section \ref{S:VP}, the eigenvalues of $\CD h_j(0)$ are much harder to be computed explicitly due to the nonlinear $f_j$.   
The following lemma follows from a very slight modification of Proposition 4.1 and Lemma 4.2 in \cite{SWZ13} and we skip its proof. (Even though the proof there was for the case $\mu_j>0$, the case of $\mu_j<0$ follows from the oddness of $f_j$.)

\begin{lemma} \label{C:Dh}
There exists $\delta>0$ such that 
\[
\big(\CD h_j(0) \big) \sum_{m=3}^\infty (a_m \cos m\theta + b_m \sin m \theta) = \sum_{m=2}^\infty m \lambda_{j, m} (b_{m+1} \cos m\theta - a_{m+1} \sin m \theta)
\]
where $\lambda_{j, m} \in \BFR$ and $|\lambda_{j, m}|\geq \delta$. 
\end{lemma}

Therefore we have 

\begin{corollary} \label{C:CF-zero}
It holds that $\CF(X_*, 0, 0)=0$ and $\big(\CD_{(X, \beta)} \CF(X_*, 0, 0)\big)^{-1} \in L\big( (\dot H^{s-1})^N, \BFR^{2N} \times (X^s)^N \big)$ if and only if $X_*= (x_1^*, \ldots, x_N^*)$ is a non-degenerate critical point of $H_{\vec{\CC}}(X)$.  
\end{corollary}

As in the last step of Section \ref{S:VP}, following from the Implicit Function Theorem, for $|\vr|<<1$, there exists $\big(X(\vr), \beta(\vr)\big) \in \Sigma_\rho \times B_{(X^s)^N, R_s}$ smooth in $\vr$, such that 
\[
\CF \big(X(\vr), \beta(\vr), \vr\big) =0, \qquad X(0) = X_*, \; \; \beta(0)=0.
\]
It yields the desired steady concentrated vorticity $\omega$ whose vortical domain boundaries $\p \Omega_j$ are $O(|\vr|r_j^2)$ perturbations to a small circle of radius $r_j$. Such $\omega$ is continuous on $\bar \Omega$ and piecewise smooth on $\bar \Omega_1, \ldots, \bar \Omega_N, \bar \Omega \backslash \cup_{j=1}^N \Omega_j$ due to the construction. Remark \ref{R:shape} also applies here.

\bibliography{VP}{}
\bibliographystyle{plain}

\end{document}